\documentclass[a4paper]{amsproc}

\usepackage{amssymb}
\usepackage[hyphens]{url} 
\usepackage[dvips]{graphicx} 

\theoremstyle{plain}
 \newtheorem{thm}{Theorem}[section]
 \newtheorem{prop}{Proposition}[section]
 \newtheorem{lem}{Lemma}[section]
 \newtheorem{cor}{Corollary}[section]
\theoremstyle{definition}
 \newtheorem{exm}{Example}[section]
 \newtheorem{dfn}{Definition}[section]
\theoremstyle{remark}
 \newtheorem{rem}{Remark}[section]
 \numberwithin{equation}{section}

\newtheorem{THEO}{Theorem}

\renewcommand{\leq}{\leqslant}
\renewcommand{\geq}{\geqslant}

\title[Orthogonality of quasi-orthogonal polynomials]{ORTHOGONALITY OF QUASI-ORTHOGONAL POLYNOMIALS}


 
\author[Bracciali]{\bfseries Cleonice F. Bracciali} 
\address{Departamento de Matem\'{a}tica Aplicada \\
UNESP-Univ Estadual Paulista \\
S\~{a}o Jos\'{e} do Rio Preto, SP \\
Brazil}
\email{cleonice.bracciali@unesp.br}

\author[Marcell\'{a}n]{\bfseries Francisco Marcell\'{a}n} 
\address{Departamento de Matem\'{a}ticas \\
Universidad Carlos III de Madrid \\ 
Legan\'{e}s \\ 
Spain \\
and Instituto de Ciencias Matem\'{a}ticas (ICMAT) \\ 
Cantoblanco \\
Spain}
\email{pacomarc@ing.uc3m.es}

\author[Varma]{\bfseries Serhan Varma} 
\address{Department of Mathematics, Faculty of Science  \\
Ankara University \\
Tando\u{g}an \\
Ankara \\ 
Turkey}
\email{svarma@science.ankara.edu.tr}

\thanks{Research supported by the Brazilian Science Foundation CAPES under project number CSF/PVE 107/2012. The work of the author (CFB) has also been supported by Brazilian Science Foundation CNPq, grants 305208/2015-2 and 402939/2016-6. The work of the author (FM) has also been supported by Ministerio de Econom\'{\i}a y Competitividad of Spain, grant MTM 2015-65888-C4-2-P.} 

\begin{document}

\setcounter{page}{1} \thispagestyle{empty}

\maketitle

\begin{abstract}
A result of P\'olya states that every sequence of quadrature formulas $Q_n(f)$ with $n$ nodes and positive  numbers converges to the integral $I(f)$ of a continuous function $f$ provided $Q_n(f)=I(f)$ for a space of algebraic polynomials of certain degree that depends on $n$. The classical case when the algebraic degree of precision is the highest possible is well-known and the quadrature formulas are the Gaussian ones whose nodes coincide with the zeros of the corresponding orthogonal polynomials and the  numbers are expressed in terms of the so-called kernel polynomials. In many cases it is reasonable to relax the requirement for the highest possible degree of precision in order to gain the possibility to either approximate integrals of more specific continuous functions that contain a polynomial factor or to include additional fixed nodes. The construction of such quadrature processes is related to quasi-orthogonal polynomials. Given a sequence $\left \{ P_{n}\right \}_{n\geq0}$ of monic orthogonal polynomials and a fixed integer $k$, we establish necessary  and sufficient conditions so that the quasi-orthogonal polynomials $\left \{ Q_{n}\right \}_{n\geq0}$ defined by
\[
Q_{n}(x) =P_{n}(x) + \sum \limits_{i=1}^{k-1} b_{i,n}P_{n-i}(x), \  \ n\geq 0,
\]   
with $b_{i,n} \in \mathbb{R}$, and $b_{k-1,n}\neq 0$ for $n\geq k-1$, also constitute a sequence of orthogonal polynomials. Therefore we solve the inverse problem for linearly related orthogonal polynomials.
The characterization turns out to be equivalent to some nice recurrence formulas for the coefficients $b_{i,n}$. We employ these results to establish explicit relations between various types of quadrature rules from the above relations. A number of illustrative examples are provided.
\end{abstract}

\section{Introduction} 

Some results obtained during the early development of the theory of orthogonal polynomials were motivated by the desire to build quadrature formulas with positive Christoffel numbers whose nodes are zeros of known polynomials. Nowadays  these quadratures are succinctly denominated as positive quadrature formulas. The study of this kind of problems was inspired by the Gauss' theorem on quadrature with the highest algebraic degree of precision with nodes at the zeros of
the polynomials orthogonal with respect to the measure of integration as well as by the result of  P\'{o}lya \cite{Pol} on convergence of quadrature rules.
This led Riesz, Fej\'{e}r and Shohat to search for the properties of certain linear combinations of orthogonal polynomials
and the further developments resulted in deep outcome. The most convincing example is the Askey and Gasper \cite{AskGas1,AskGas2} proof of the positivity of certain sums of Jacobi polynomials which played
a key role in the final stage of de Branges' proof of the Bieberbach conjecture. We refer to the nice survey of Askey  \cite{Ask} for the motivation to study positive Jacobi polynomial sums, coming from positive quadratures, and for further information about these natural  connections.

The construction of positive quadrature rules is connected with the so-called quasi-orthogonal polynomials. Let $\left \{ P_{n}\right \}_{n\geq0}$ be a given sequence of monic orthogonal polynomials, generated by the three-term  recurrence relation
\begin{equation} \label{1}
x\, P_{n}(x) = P_{n+1}(x) + \beta _{n}\, P_{n}(x)  + \gamma _{n}\, P_{n-1} (x),\ \ n\geq 0,\ \ \gamma _{n}\neq0,
\end{equation}
with $P_{-1}(x) = 0$ and $P_{\,0}(x)  =  1.$
Then, given $k\in \mathbb{N}$, the polynomials defined by
\begin{equation} \label{2}
Q_{n}(x) = P_{n}(x) + \sum_{i=1}^{k-1}b_{i,n}P_{n-i}(x), \ \ {\rm for} \ n \geq k, 
\end{equation}
are said to be a sequence of quasi-orthogonal polynomials of order $k-1$ or, simply, $(k-1)$-quasi-orthogonal polynomials if $b_{k-1,n}\neq 0$ . Here $b_{i,n}$ for $n \geq 0$,
are real  numbers. By convention we set $b_{0,n}=1$, \ $b_{-1,n}=b_{-2,n}=0$, \  $b_{i,n}=0$ when $i>n$, and also $b_{i,n} = 0$ when $n \geq k$ and $i \geq k$.
Notice that for $k=1$ we have the standard orthogonality. This notion was introduced by Riesz while studying the moment
problem and the reason for this nomenclature is rather simple: $Q_n$ is orthogonal to every polynomial
of degree not exceeding $n-k$ with respect to the functional of orthogonality of $\{ P_n \}_{n\geq0}$. M. Riesz himself considered only the case $k=2$ while Fej\'er  \cite{Fej} concentrated his attention on the specific case when $k=3$, $P_n$ are the Legendre polynomials and $b_{2,n} <0$. It seems that Shohat \cite{31} was the first who studied the general case. The renewed recent interest on the quasi-orthogonal polynomials brought a large number of interesting results. Peherstorfer \cite{PehMC,29,PehCam}  and Xu \cite{36} obtained results concerning the location of the zeros of the quasi-orthogonal polynomials and the positivity of the Christoffel numbers when $\{ P_n \}_{n\geq0}$ are orthogonal on $[-1,1]$ with respect to a measure that belongs to Szeg\H{o}'s class. Xu \cite{37} established general properties of quasi-orthogonal polynomials and, under the assumption that
$Q_n$ is also orthogonal, studied the relation between the Jacobi matrices associated with both sequences. The zeros of some quasi-orthogonal polynomials were studied recently by Beardon and Driver \cite{6} and Brezinski, Driver and Redivo-Zaglia \cite{10}.

Motivated by the relation between positive quadrature rules and quasi-or\-tho\-go\-nal polynomials, we provide necessary and sufficient conditions in order that the sequence of polynomials $\{ Q_n \}_{n\geq0}$, obeying  (\ref{2}), is also orthogonal. The latter problem is purely algebraic in nature. We solve it via a constructive approach by taking into account classical results on Sturm sequences. It becomes evident then that one may look at the solution in terms of a relation between the Jacobi matrices associated with the sequences of orthogonal polynomials.  As a result the solution is explicit in the sense that we establish the connection between the three term recurrence relations that generate the sequences $\{ P_n \}_{n\geq0}$ and $\{ Q_n \}_{n\geq0}$ as well as between the linear functionals related to them. These results allow us to judge about the nodes of two Gaussian type quadrature formulas whose location coincides with the zeros of the polynomials $P_n$ and $Q_n$. Moreover, the Christoffel numbers of the quadrature rules are obtained explicitly as a consequence of the closed forms of the corresponding kernel polynomials which are also derived from our general approach.

The structure of the paper is as follows. In Section 2 we state the necessary and sufficient conditions of the orthogonality of a sequence of quasi-orthogonal polynomials of order $k-1$ as well as the expression of the polynomial $h$ associated with the Geronimus transformation of the initial linear functional. In Section 3, the proofs of those theorems are given as well as an algorithm to deduce the sequence of connection coefficients. Section 4 is focussed on the relation between the corresponding Jacobi matrices. Thus, we have a computational approach to the zeros of $Q_{n}(x)$ since they are the eigenvalues of the $n$th principal leading submatrices of the corresponding Jacobi matrix. The Christoffel numbers are their normalized eigenvectors. We also prove some results concerning the zeros of the polynomial $Q_{n}(x)$ as well as the expression of the kernel polynomials in terms of the initial ones. In Section 5 we analyze some examples illustrating the problems considered in the previous sections. First, the case when $u$ is a symmetric linear functional is considered. The results are implemented for Chebyshev polynomials of the second kind. Second, the non-symmetric case is studied and implemented for Laguerre polynomials. Finally, we study the case of constant coefficients. In such a case, we solve a problem posed in \cite{3} for $k\geq 3$ in such a way in a symmetric case, periodic sequences for the parameters of the three term recurrence relation appear.

\section{Orthogonality of quasi-orthogonal polynomials}

The characterization of those quasi-orthogonal polynomials  (\ref{2}) which form a sequence of orthogonal polynomials themselves  can be approached from a general point of view.
Let $\mathcal{P}$ be the linear space of algebraic polynomials with complex coefficients. Then
$\left\langle u,f\right \rangle $ denotes the action of the linear
functional $u\in \mathcal{P}^{\prime}$ over the polynomial $f\in \mathcal{P},$ where
$\mathcal{P}^{\prime }$ denotes the algebraic dual of the linear
space $\mathcal{P}$. The sequence of monic orthogonal polynomials (SMOP) $\left \{ P_{n}\right \}_{n\geq0}$ with respect to the linear functional $u$ obeys the conditions $\left \langle u,P_{n}P_{m}\right \rangle =K_{n}\delta _{nm}$, where $K_{n}\neq 0$ for all $n\geq 0$, and $\delta _{nm}$ is the Kronecker delta. A linear functional $u$ is said to be regular or quasi-definite (see \cite{14}) when the leading principal submatrices $H_{n}$ of the Hankel matrix $H=\left( u_{i+j}\right) _{i,j\geq 0}$ composed by the moments $u_{i}=\left \langle u,x^{i}\right \rangle $, $i \geq 0$, are non-singular for each $n\geq 0$. When the determinants of $H_{n}$ are positive for all nonnegative integers $n$ the functional is called positive-definite. If  the linear functional $u$ is regular, then  the SMOP $\left \{ P_{n}\right \}_{n\geq0} $ satisfies the three-term recurrence relation (\ref{1}) with $\gamma _{n}\neq 0$ and if $u$ is positive-definite then $\gamma _{n}> 0$. Conversely, if a sequence of polynomials is generated by the recurrence relation (\ref{1}) and $\gamma _{n}\neq 0$, then there is a linear functional $u\in \mathcal{P}^{\prime }$, such that $\{ P_n \}_{n\geq0}$ is a sequence of polynomials orthogonal with respect to $u$ and this is the statement of Favard's theorem (\cite{14}). Moreover, if $\gamma _{n}> 0$ for every $n\in \mathbb{N},$ then the linear functional $u$ is positive-definite and it has  an integral representation $\left \langle u,f\right \rangle =\int_{\mathbb{R}}fd\mu$, $f\in \mathcal{P}$, where $d \mu$ is a positive Borel measure supported on an infinite subset of $\mathbb{R}$ (see \cite{14}).

The linear functional $v\in \mathcal{P}^{\prime}$ is called a rational perturbation of $u\in \mathcal{P}^{\prime}$, if there exist polynomials $p$ and $q$, such that
\begin{equation*}
q (x) v= p(x) u.
\end{equation*}
Detailed information about the direct problems studied from  several points of view can be found in \cite{2,11,18,25,38}.
In particular, the connection formula between the polynomials orthogonal with respect to $v$ and $u$ is called the generalised Christoffel's formula (see \cite{18}).
The relation between the corresponding Jacobi matrices was studied in \cite{17}.

Let $\{P_{n}\}_{n\geq0}$  be a SMOP, $m$ and $k$ are positive integers. Let consider another sequence of monic polynomials $\{Q_{n}\}_{n\geq0} $  related to $\{P_{n}\}_{n\geq0}$  by
\begin{equation}
Q_{n}(x) +\sum \limits_{j=1}^{m-1} a_{j,n} Q_{n-j}(x)
=P_{n}(x) +\sum \limits_{i=1}^{k-1}  b_{i,n} P_{n-i}(x),\ \ n\geq 0,
\label{GenQ}
\end{equation}
with $a_{j,n}, b_{i,n} \in \mathbb{R}$, $ a_{m-1,n}\,  b_{k-1,n} \neq0$.
Then the problem to find necessary and sufficient conditions so that $\{Q_{n}\}_{n\geq0}$ is also a SMOP and to obtain the relation between the corresponding regular linear functionals is called an inverse problem. Observe that we adopt the convention that when either $m$ or $k$ is equal to one, then the corresponding sum does not appear, that is, we interpret it as an empty one. A vast number of interesting results have been obtained on topics related to the inverse problem (see \cite{1,3,4,5,8,9,21,22,27,30}).

In the present contribution we also focus our attention on the quasi-orthogonal polynomials defined by (\ref{2}) under the only natural restriction and $b_{k-1,n}\neq 0$ for $n\geq k-1$. This corresponds to a very general situation when we set $m=1$ and $k\in \mathbb{N}$ in (\ref{GenQ}).
Therefore, in what follows we consider this setting. Many particular results, when one looks for the relation between the functionals $u$ and $v$, with respect to which the polynomial sequences $\{P_n\}_{n\geq0}$ and $\{Q_n\}_{n\geq0}$ are orthogonal, are known \cite{11,12,15,16,24,38}
but the general case that we discuss in the present contribution has not been approached in the literature yet.   In this paper we provide necessary and sufficient conditions so that the sequence of monic polynomials $\{Q_{n}\}_{n\geq0}$ is also orthogonal.

Let $\left \{ P_{n}\right \} _{n\geq 0}$ be a SMOP corresponding to a
regular linear functional $u$. Now we give the necessary and
sufficient conditions ensuring the orthogonality of the monic polynomial
sequence $\left \{ Q_{n}\right \} _{n\geq 0}$ that satisfies the three-term recurrence relation
\begin{equation*} 
xQ_{n}(x) =Q_{n+1}(x) + \tilde{\beta}_{n}Q_{n}\left(x\right) +\tilde{\gamma}_{n}Q_{n-1}(x),\ \ n\geq 0,
\end{equation*}
with the initial conditions $Q_{-1}(x) =0$ and $Q_{0}\left(x\right) =1$, and the condition $\tilde{\gamma}_{n} \neq 0$, for $n\geq1$.

\begin{thm}
\label{maintheorem}
Let $\left\{ Q_{n}\right\} _{n\geq 0}$ be a sequence of monic polynomials
defined by $\left( \ref{2}\right) $. Then $\left\{ Q_{n}\right\} _{n\geq
0}$ is a SMOP with recurrence coefficients $\{ \tilde{\beta}_{n}\}_{n\geq 0}$ and $\{ \tilde{\gamma}_{n}\} _{n\geq 1}$
if and only if the coefficients $b_{0,n}=1$, $\{b_{i,n}\}_{n\geq 1}$, $ 1 \leq i \leq k-1$ , satisfy the following conditions
\begin{equation}
\label{not_null}
\gamma_{n} + b_{2,n} - b_{2,n+1} + b_{1,n} \left(  \beta_{n-1}-\beta_{n}-b_{1,n}+b_{1,n+1}\right) \neq 0, \ \ {\rm for} \ n \geq 1,
\end{equation}
\begin{equation}
\label{eq_for_b1}
b_{1,n+1} = b_{1,n} + \beta _{n} - \beta _{n-k+1}  + \frac{b_{k-2,n-1}}{b_{k-1,n-1}}\gamma _{n-k+1} -\frac{b_{k-2,n}}{b_{k-1,n}}\gamma _{n-k+2}, \  n \geq k,
\end{equation}
\begin{equation}
\label{eq_for_b2}
b_{2,n+1} = b_{2,n} + \gamma _{n} -  \frac{b_{k-1,n}}{b_{k-1,n-1}}\gamma _{n-k+1}
 + b_{1,n}\left( \beta_{n-1}-\beta _{n}-b_{1,n}+b_{1,n+1}\right),  \  n \geq k,
\end{equation}
and
\begin{eqnarray}
\label{eq_for_b_i}
b_{i+2,n+1} & = &   b_{i+2,n}  + b_{i+1,n}\left( \beta_{n-1-i}-\beta_{n}-b_{1,n}+b_{1,n+1}\right) + b_{i,n}\gamma_{n-i} \notag\\
&   & - b_{i,n-1}\left[\gamma_{n} + b_{2,n} - b_{2,n+1} + b_{1,n} \left(  \beta_{n-1}-\beta_{n}-b_{1,n}+b_{1,n+1}\right)  \right],
\end{eqnarray}
for $1 \leq i\leq k-3$ and   $n\geq i+1$.

Moreover, the recurrence coefficients of $\left\{ Q_{n}\right\} _{n\geq 0}$ are given by
\begin{eqnarray}
\label{7} 
\tilde{\beta}_{n} &=&\beta _{n}+b_{1,n}-b_{1,n+1}, \ \ n\geq 0, \\
\label{8} 
\tilde{\gamma}_{n} &=&\gamma _{n}+b_{2,n}-b_{2,n+1}+b_{1,n}\left( \beta_{n-1}-\beta _{n}-b_{1,n}+b_{1,n+1}\right), \ \ n\geq 1,
\end{eqnarray}
and the coefficients $\tilde{\gamma}_{n}$ also satisfy
\begin{equation} \label{gamma_tilde_gamma}
\tilde{\gamma}_{n} = \frac{b_{k-1,n}}{b_{k-1,n-1}}  \gamma _{n-k+1}, \ \ n \geq k.
\end{equation}
\end{thm}

The above relations provide a complete characterization of the orthogonality of the polynomial sequence  $\{ Q_n\}_{n\geq0}$. When $b_{j,n}=b_{j}, j=1,\cdots, k-1,$ you recover Theorem 1 in \cite{3}.

On the other hand, a natural question arises about the relation between the regular  linear functionals  $u$ and $v$ such that $\{P_n\}_{n\geq0}$ and $\{ Q_n\}_{n\geq0}$ are the
corresponding SMOP. In this case, the functional $v$ which describes the orthogonality of the sequence $\{ Q_n\}_{n\geq0}$ is a Geronimus spectral transformation of degree
$k-1$ of the linear functional $u$. In other words, $u= h(x) v$, where $h$ is a polynomial of degree $k-1$ (see \cite{27}). Our next result furnishes a method to determine $h$.

\begin{thm}
\label{theorem_about_h}
The coefficients of the polynomial
\begin{equation}
h(x) = h_{0}+h_{1}x+\cdots+h_{k-2}x^{k-2}+h_{k-1}x^{k-1},
\label{poly_h}
\end{equation}
such that $u  = h(x) v$, are the unique solution of a system of $k$ linear equations, where the entries of the corresponding matrix depend only on the sequences of connection coefficients $\{b_{i,n}\}_{n\geq k-1}$, $i= 1,2, \dots, k-1$.
\end{thm}

A detailed description of the linear system and about the explicit form of the coefficients will be done in the sequel.

It is worth pointing out that an alternative way to compute the coefficients of $h$ is via a relation between the Jacobi matrices related to the sequences  $\{ P_n\}_{n\geq0}$ and $\{ Q_n\}_{n\geq0}$. We discuss this method in Section \ref{section_matrices}.

Since the quasi-orthogonal polynomials arise naturally in the context of quadrature formulae of Gaussian type, many properties that can be classified more than as analytic rather than algebraic, such as the behaviour of their zeros and the positivity of the Christoffel numbers have been analysed. Most of these results deal with rather specific particular cases when either $k$ is a small integer or the orthogonal polynomials belong to classical families. In Section \ref{section_zeros} we obtain some results about the zeros of the polynomials $P_n$ and $Q_n$.

Many illustrative examples are analysed when the linear functional $u$ is a symmetric one, as well as when one deals with constant connection coefficients. The latter problem is motivated by a result in \cite{Grinshpun2004} where  $\{ P_n\}_{n\geq0}$ is the sequence of Chebyshev polynomials.

\section{Proofs of Theorems \ref{maintheorem} and \ref{theorem_about_h} and the direct problem}

\subsection{Proof of Theorem \ref{maintheorem}}

The core of the overall approach is a classical result of Sturm \cite{Sturm35} on counting the number of real zeros of an algebraic polynomial.  We refer to \cite[Section 10.5]{Rah} and \cite[Sections 2.4, 2.5]{Obr} for detailed information about various versions of Sturm's result as well as about the historical background. We state the general version of Sturm's  theorem in the setting we need. Let $R_{n+1}$ and $R_n$ be polynomials of exact degree $n+1$ and $n$, respectively, with monic leading coefficients. Execute the Euclidean algorithm
\begin{equation} \label{EA}
R_{k+1}(x) = (x-c_{k}) R_{k}(x) - d_{k} R_{k-1}(x),\ \ k=n,n-1,\ldots, 1.
\end{equation}
A careful inspection of the general version of Sturm's theorem shows that the following holds:
\begin{THEO} {\rm (Sturm)}
\label{ThS}
Under the above assumptions, the polynomials $R_{n+1}$ and $R_{n}$ have real and strictly interlacing zeros if and only if
 $d_{k}, \  k=n,n-1,\ldots, 1,$  are positive real numbers. Furthermore, the zeros of the polynomial $R_{k}, \  k=n,n-1,\ldots, 1,$  are all real and the zeros of two consecutive polynomials are strictly interlacing.
\end{THEO}

It follows immediately from Theorem \ref{ThS} and Favard's theorem that, given two polynomials $R_{n+1}$ and $R_{n}$ with positive leading coefficients and with real and strictly interlacing zeros, the Euclidean algorithm  (\ref{EA}) generates the sequence $R_k$, $k=0,\ldots, n+1$, such that these are the first $n+1$ terms of a sequence of orthogonal polynomials, which can be constructed by using the standard three term recurrence relation.  In other words, any two polynomials of consecutive degrees and interlacing zeros may be ``embedded" in a sequence of orthogonal polynomials. This straightforward but beautiful observation was pointed out by  Wendroff \cite{33} and the statement is nowadays called Wendroff's theorem. Observe that $R_{n+1}$ and $R_n$ generate $R_k$, $k=n-1,\ldots,0$ uniquely ``backwards" via (\ref{EA}) while the sequence $R_k$, $k=0,\ldots,n+1$ of all the polynomials can be extended ``forward" in  various ways. The complete characterization of the sequences of orthogonal polynomials $P_n$ and $Q_n$ that are related by the relation (\ref{2}) is obtained via Theorem \ref{ThS}. \\

{\bf Proof of Theorem \ref{maintheorem}}.

Applying the Euclidean algorithm (\ref{EA}) with ``initial'' polynomials  $R_{n+1}(x) = Q_{n+1}(x)$ and $R_{n}(x)= Q_{n}(x)$ and setting  $c_n = \tilde{\beta}_{n}$, we obtain
\begin{equation*}
Q_{n+1}(x) = (x-\tilde{\beta}_{n}) Q_{n}(x) - R_{n-1}(x),
\end{equation*}
where $R_{n-1}(x)$ is a polynomial of degree at most $n-1$. Using (\ref{2}) together with the recurrence relation (\ref{1}) we conclude that
\begin{equation} \label{rel_R1}
R_{n-1}(x) = \sum_{i=0}^{k}
\left[ b_{i,n}(\beta_{n-i}-\tilde{\beta}_{n} ) - b_{i+1,n+1} + b_{i+1,n} + b_{i-1,n}\gamma_{n-(i-1)} \right]P_{n-i}(x),
\end{equation}
where  $b_{-1,n}=0$ and $b_{0,n}=1$. Moreover, when $n \geq k$, we have $b_{i,n} = 0$ for all $i \geq k$.

Now we can determine necessary and sufficient conditions in order to
the polynomial $R_{n-1}(x)$ coincides with the polynomial $\tilde{\gamma}_{n} Q_{n-1}(x)$, i.e.,
\begin{equation} \label{rel_R2}
R_{n-1}(x) =  \tilde{\gamma}_{n} \left( P_{n-1}(x) +\sum
\limits_{i=1}^{k-1}b_{i,n-1}P_{n-1-i}(x) \right).
\end{equation}

Comparing the coefficients that multiply $P_{n}(x)$ and $P_{n-1}(x)$ in (\ref{rel_R1}) and (\ref{rel_R2}) we derive the conditions
\begin{eqnarray*}
\beta_{n} - \tilde{\beta}_{n}  - b_{1,n+1} + b_{1,n}                    & = & 0, \ \ n \geq 0, \\
b_{1,n}(\beta_{n-1}-\tilde{\beta}_{n}) - b_{2,n+1}+b_{2,n} +\gamma_{n} & = & \tilde{\gamma}_{n}, \ \ n \geq 1,
\end{eqnarray*}
and the latter obviously correspond to (\ref{7}) and (\ref{8}). This means that
\begin{equation}  \label{3}
 \tilde{\gamma}_{n} = \gamma _{n}+b_{2,n}-b_{2,n+1}+b_{1,n}\left( \beta_{n-1}-\beta _{n}-b_{1,n}+b_{1,n+1}\right), \ \ n\geq 1.
\end{equation}
Since $\tilde{\gamma}_{n} \neq 0$, we obtain the constraint
\begin{equation*}
\gamma_{n} + b_{2,n} - b_{2,n+1} + b_{1,n} \left(  \beta_{n-1}-\beta_{n}-b_{1,n}+b_{1,n+1}\right) \neq 0, \ \ {\rm for} \ n \geq 1,
\end{equation*}
which is exactly (\ref{not_null}).

Similarly, comparing the coefficients of $P_{n-2}(x),...,P_{n-k}(x)$ in (\ref{rel_R1}) and (\ref{rel_R2}), we obtain the following conditions:
\begin{eqnarray}
b_{i,n-1}\tilde{\gamma}_{n} & = & b_{i,n}\gamma_{n-i}+b_{i+2,n}-b_{i+2,n+1}+b_{i+1,n}\left( \beta _{n-1-i}-\beta_{n}-b_{1,n}+b_{1,n+1}\right) \text{,}  \notag \\
&  & 1 \leq i\leq k-3, \ \ n\geq i+1, \label{4} \\
b_{k-2,n-1}\tilde{\gamma}_{n} & = & b_{k-2,n}\gamma _{n-k+2}+b_{k-1,n}\left(
\beta _{n-k+1}-\beta _{n}-b_{1,n}+b_{1,n+1}\right),  \notag \\
&  &  n\geq k-1      \label{5}
\end{eqnarray}
and
\begin{eqnarray}
b_{k-1,n-1}\tilde{\gamma}_{n} & = & b_{k-1,n}\gamma _{n-k+1}, \ \ n\geq k.
\label{6}
\end{eqnarray}%

Now (\ref{eq_for_b1}) follows from  (\ref{5}) and (\ref{6}) while (\ref{eq_for_b2}) is a consequence of  (\ref{3}) and (\ref{6}).
Finally, (\ref{3}) and (\ref{4}) imply (\ref{eq_for_b_i}).

It is important to check that at the last step the coefficient $b_{k-1,n+1}$ must be different from zero in order to be consistent with the quasi-orthogonality condition.
This completes the proof.

Theorem \ref{maintheorem} provides also a forward algorithm to compute the coefficients $b_{i,n}$ for $n \geq k+1$.
Starting with coefficients $b_{i,k-1}$, $i=1,2,\ldots,k-1,$ from the linear combination
\begin{equation*}
Q_{k-1}(x) =  P_{k-1}(x) + b_{1,k-1}P_{k-2}(x) + \cdots + b_{k-1,k-1}P_{0}(x),
\end{equation*}
we choose the coefficients $b_{i,k}$ for $i=1,2,\ldots,k-1,$ and write
\begin{equation*}
Q_{k}(x) =  P_{k}(x) + b_{1,k}P_{k-1}(x) + \cdots + b_{k-1,k}P_{1}(x).
\end{equation*}
Then we compute $b_{1,n+1}$, for $n \geq k$, using equation (\ref{eq_for_b1}) and $b_{1,n}$, $b_{k-2,n-1}$, $b_{k-1,n-1}$, $b_{k-2,n}$ and $b_{k-1,n}$ (see the first scheme in Fig.~\ref{fig:1}).
We compute $b_{2,n+1}$, for $n \geq k$, using equation (\ref{eq_for_b2}) and $ b_{2,n}$, $b_{1,n}$, $b_{1,n+1}$, $b_{k-1,n}$ and $b_{k-1,n-1}$ (see the second scheme in Fig.~\ref{fig:1}).

We compute $b_{i+2,n+1}$, for $n \geq k$ and $1 \leq i \leq k-3 $, using equation (\ref{eq_for_b_i}) and
$b_{i+2,n}$, $b_{i+1,n}$, $b_{i,n}$, $b_{i,n-1}$,  and also $b_{1,n}$, $b_{1,n+1}$, $b_{2,n}$, and $b_{2,n+1}$. This is illustrated as the first scheme in  Fig.~\ref{fig:2}.
Alternatively, $b_{i+2,n+1}$, for $n \geq k$ and $1 \leq i \leq k-3 $, is given by
\begin{eqnarray*}
b_{i+2,n+1} & = &   b_{i+2,n} + b_{i+1,n}\left( \beta_{n-1-i}-\beta_{n}-b_{1,n}+b_{1,n+1}\right) + b_{i,n}\gamma_{n-i}  \\
&  & -  b_{i,n-1}\frac{b_{k-1,n}}{b_{k-1,n-1}}\gamma _{n-k+1},
\end{eqnarray*}
using $b_{i+2,n}$, $b_{i+1,n}$, $b_{i,n}$, $b_{i,n-1}$,  and also $b_{1,n}$, $b_{1,n+1}$, $b_{k-1,n-1}$, $b_{k-1,n}$,
(see the second scheme in Fig.~\ref{fig:2}).

\begin{figure}[htb]
 \includegraphics[width=5.5cm]{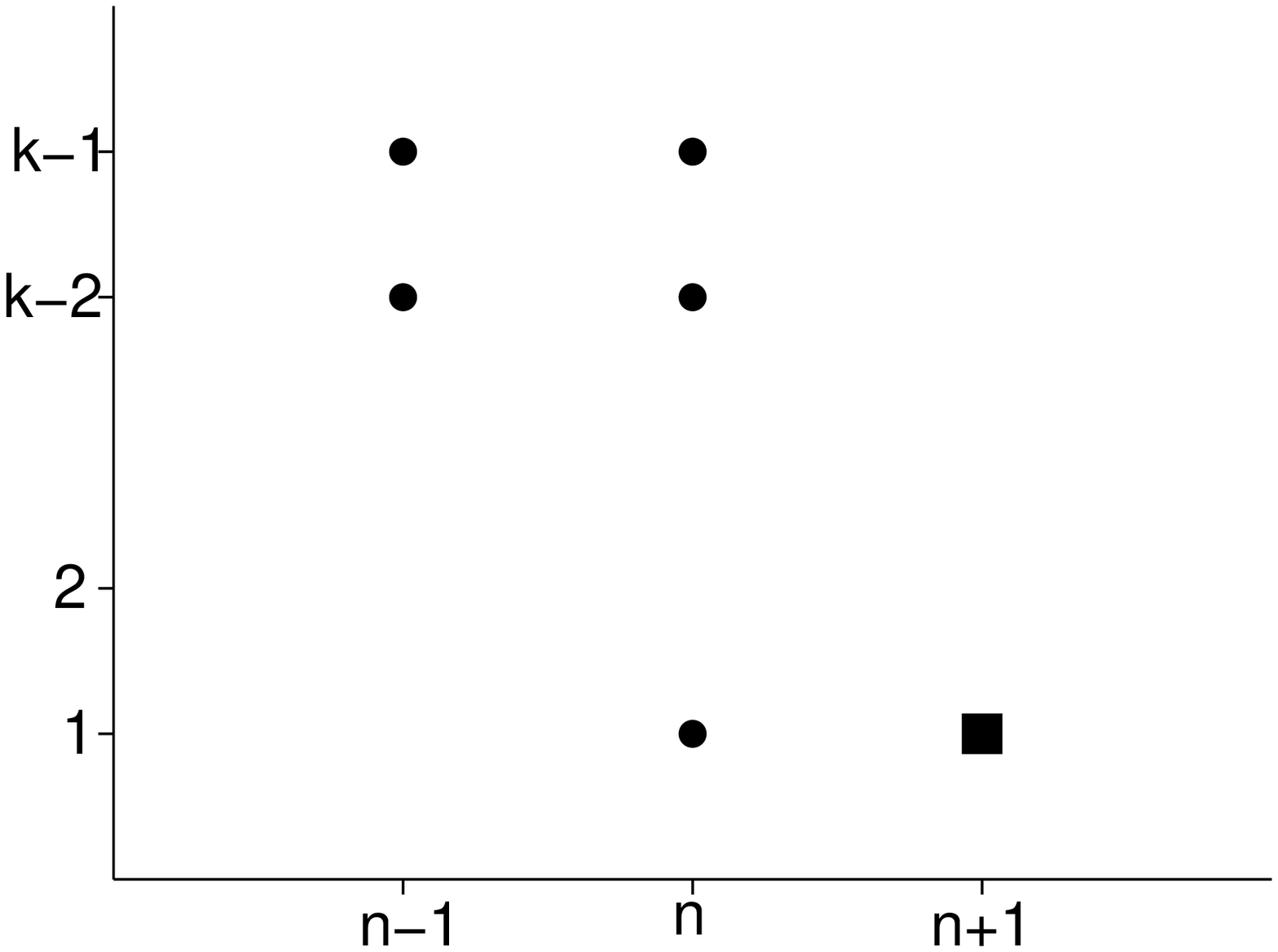}
 \includegraphics[width=5.5cm]{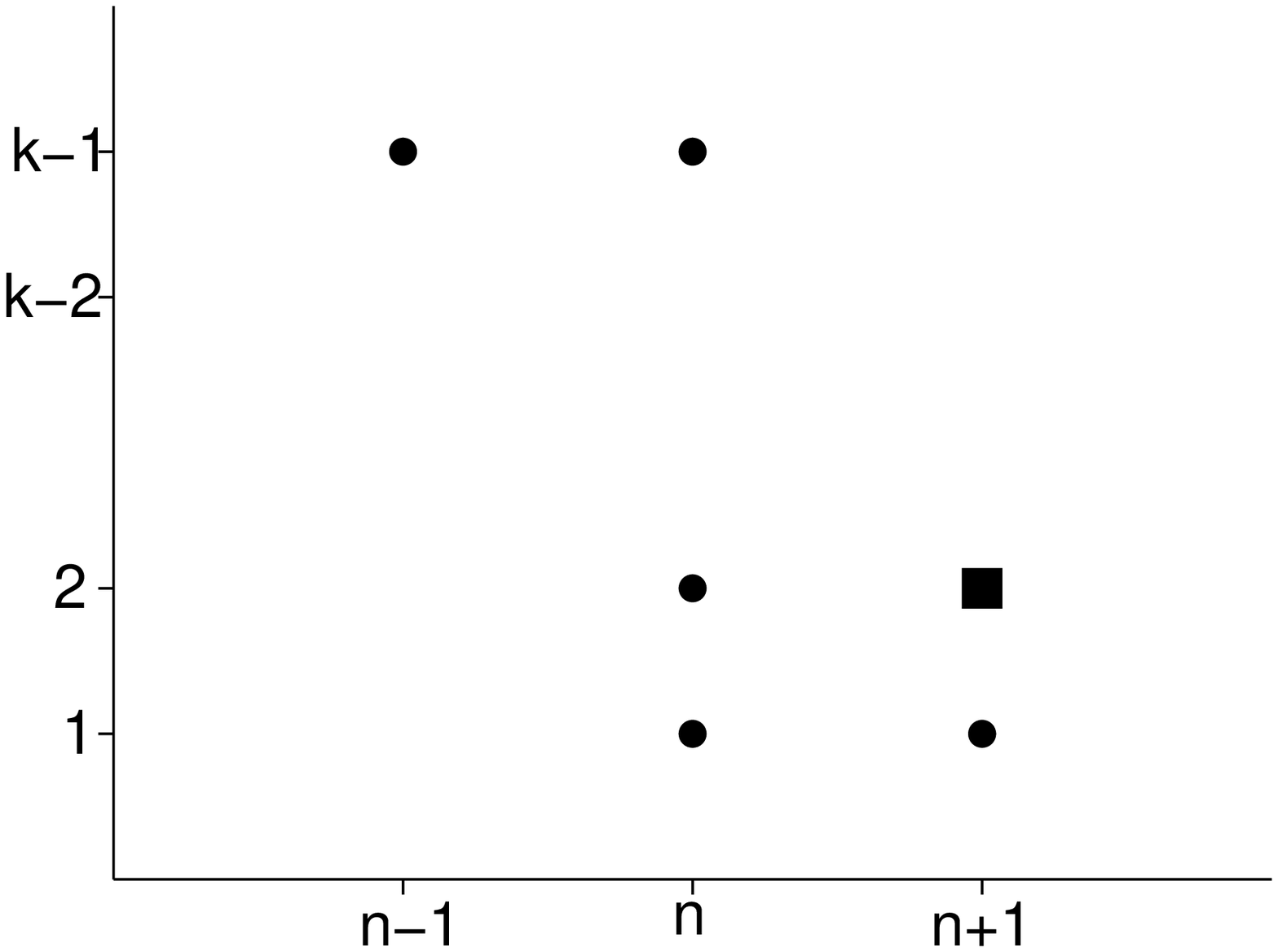}
\caption{Scheme for the calculation of $b_{1,n+1}$ and $b_{2,n+1}$, $n \geq k$.}
\label{fig:1}
\end{figure}

\begin{figure}[htb]
 \includegraphics[width=5.5cm]{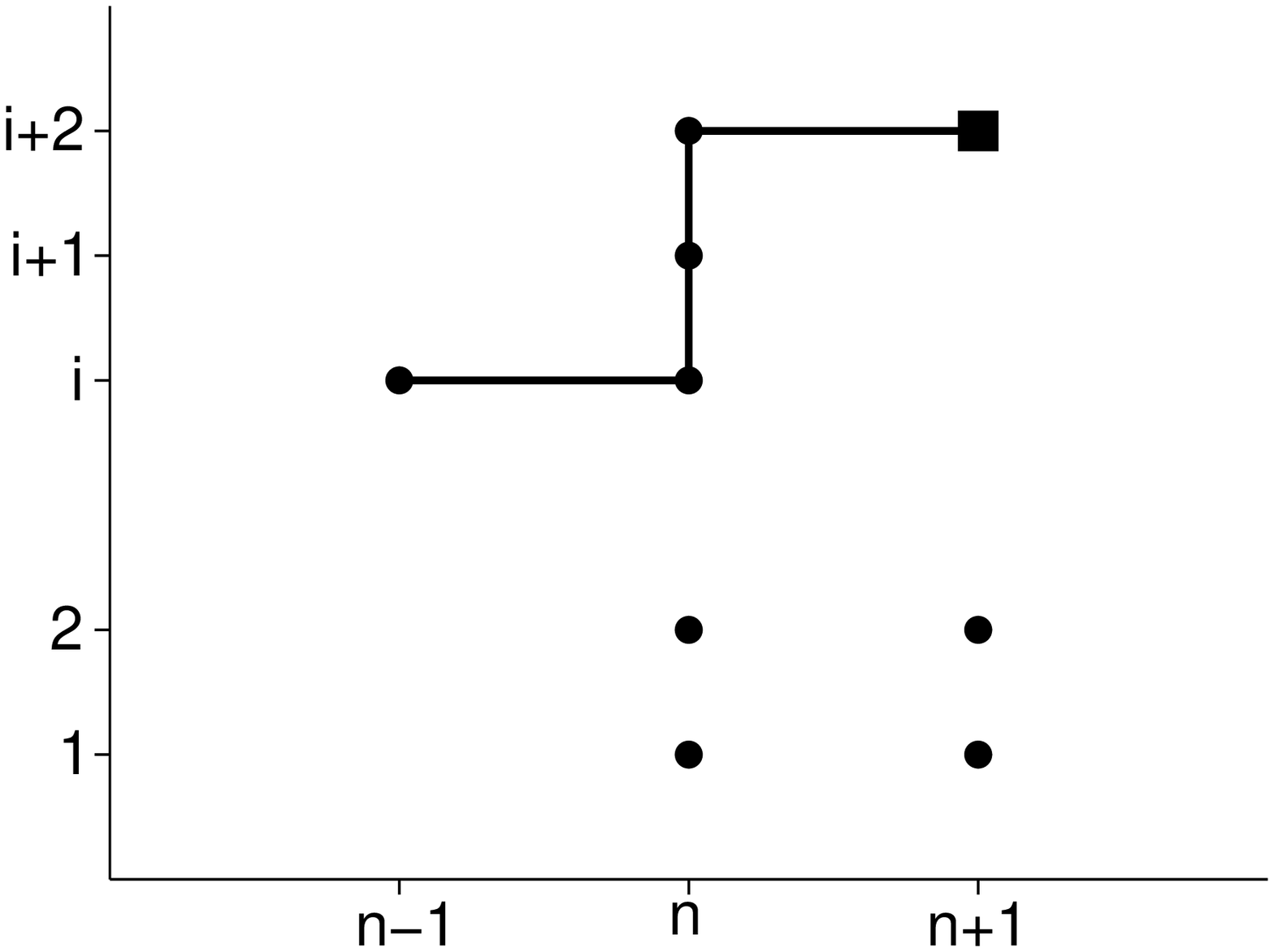}
 \includegraphics[width=5.5cm]{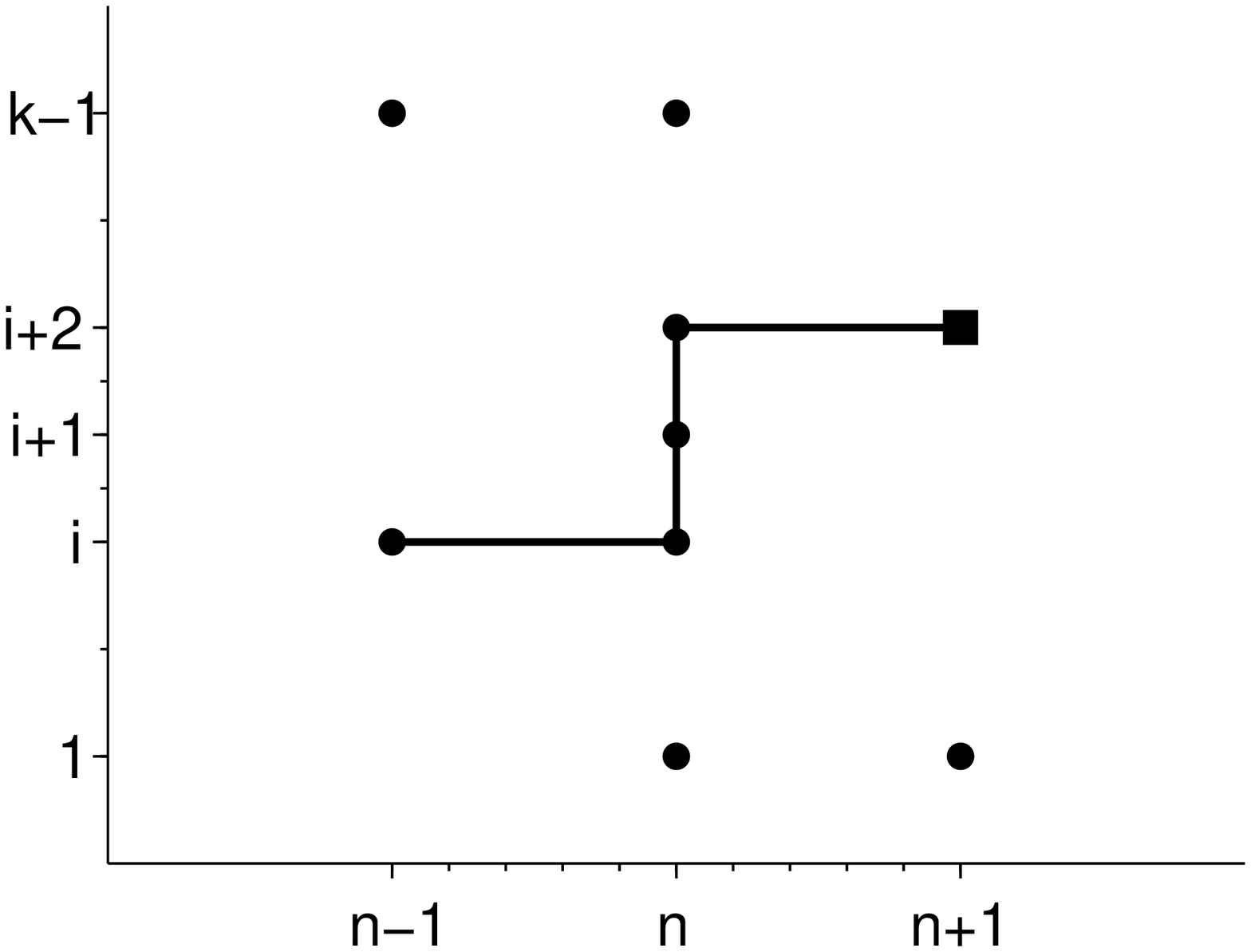}
\caption{Alternative schemes for calculation of $b_{i+2,n+1}$, $n \geq k$.}
\label{fig:2}
\end{figure}

As we have pointed out above, after the computations at level $n+1$, it is necessary to verify if $b_{k-1,n+1} \neq 0$, for $n \geq k$.

The initial coefficients  $b_{0,n}=1$, $b_{1,n}, b_{2,n}, \ldots, b_{n,n}$, for $1 \leq n \leq k-2$, starting from $Q_{k}$ and $Q_{k-1}$, are uniquely determined by the ``backward" process
described by the Euclidean algorithm and by Theorem \ref{ThS}.

Let us notice the key role played by the connection coefficients for the  polynomials $Q_{k-1}$ and  $Q_{k}$  as initial data to run the above algorithm.

As a summary, you can generate the coefficients of quasi-orthogonal polynomials in a recursive way, assuming some initial conditions.

\subsection{Proof of Theorem \ref{theorem_about_h}}

The dual basis $\left \{ \omega _{n}\right \} _{n\geq 0}\in \mathcal{P}^{^{\prime }}$ of $\left \{ P_{n}\right \} _{n\geq 0}$ is defined, as usual,  by the conditions (see \cite{Maroni1991})
\begin{equation*}
\left \langle \omega _{n},P_{m}\right \rangle =\delta _{nm}.
\end{equation*}%
It is easy to see that the elements of  the basis, dual to SMOP $\left \{ P_{n}\right \}
_{n\geq 0}$ with respect to the regular linear functional $u$, are $\omega _{n}=\frac{P_{n}u}{\left \langle u,P_{n}^{2}\right \rangle }$. Let us
define the left-multiplication of a linear functional $u\in \mathcal{P}^{^{\prime }}$ by any polynomial $f\in \mathcal{P}$ via
\begin{equation*}
\left \langle fu,p\right \rangle =\left \langle u,fp\right \rangle, \
\ \ p\in \mathcal{P}.
\end{equation*}%
Let $\left \{ Q_{n}\right \} _{n\geq 0}$, given by relation $\left( \ref{2}\right) $, be a SMOP with respect to a regular linear functional $v$. According to \cite{Maroni1991},
if we use the expansion of the linear functional $ u $ in terms of the dual basis $\{ \frac{Q_{j}v}{\left \langle
v,Q_{j}^{2}\right \rangle } \} _{j\geq 0}$ of the SMOP $\left \{
Q_{n}\right \} _{n\geq 0}$, in view of orthogonality properties and
relation $\left( \ref{2}\right) $, we obtain the following relation between the
corresponding linear functionals.

\begin{lem}
\begin{eqnarray}
u = \sum \limits_{j=0}^{k-1}\frac{\left \langle u,Q_{j}\right \rangle }{%
\left \langle v,Q_{j}^{2}\right \rangle }Q_{j}v, \ \ \ \ \text{i.e.}, \ \ \ \
 u = h(x) v,  \label{12}
\end{eqnarray}%
where $h(x) = h_{k-1}x^{k-1}+ h_{k-2}x^{k-2} + \cdots +h_{1}x+  h_{0}$ is a
polynomial of degree $(k-1)$ because its leading coefficient is  $h_{k-1} = \frac{b_{k-1,k-1} \left \langle u, 1 \right \rangle }{\left \langle v, Q^{2}_{k-1} \right \rangle} \neq 0.$
\end{lem} 
\pagebreak

{\bf Proof of Theorem \ref{theorem_about_h}}

For $n \geq k$ we have
\begin{eqnarray*}
\left \langle u, P_{m}Q_{n}\right \rangle & = & \left \langle h(x)v,  P_{m}Q_{n}\right \rangle \\
& = & h_{0} \left \langle v,  P_{m}Q_{n}\right \rangle +
      h_{1} \left \langle v, x P_{m}Q_{n}\right \rangle + \cdots +
      h_{k-1} \left \langle v, x^{k-1} P_{m}Q_{n}\right \rangle.
\end{eqnarray*}%
For $m=n,n-1, \ldots, n-(k-1)$, we obtain
\begin{eqnarray} \label{equations}
\left \langle u, P_{n}Q_{n}\right \rangle & = &
h_{0}   \left \langle v,        P_{n}Q_{n}\right \rangle +
h_{1}   \left \langle v, x      P_{n}Q_{n}\right \rangle + \cdots +
h_{k-1} \left \langle v, x^{k-1}P_{n}Q_{n}\right \rangle \notag \\
\left \langle u, P_{n-1}Q_{n}\right \rangle & = &
h_{0}   \left \langle v,        P_{n-1}Q_{n}\right \rangle +
h_{1}   \left \langle v, x      P_{n-1}Q_{n}\right \rangle + \cdots +
h_{k-1} \left \langle v, x^{k-1}P_{n-1}Q_{n}\right \rangle \notag \\
\vdots \ \ \ \ \ \ & & \ \ \ \ \ \ \ \vdots  \\
\left \langle u, P_{n-(k-1)}Q_{n}\right \rangle & = &
h_{0}   \left \langle v,        P_{n-(k-1)}Q_{n}\right \rangle +
h_{1}   \left \langle v, x      P_{n-(k-1)}Q_{n}\right \rangle + \cdots + \notag \\
\ & \ & + h_{k-1} \left \langle v, x^{k-1}P_{n-(k-1)}Q_{n}\right \rangle. \notag
\end{eqnarray}

Since, for $j=0,1,\ldots,k-1,$
\begin{equation*}
 \left \langle v, x^{l} P_{n-j}Q_{n}\right \rangle =
 \left\{
 \begin{array}{ll}
  0,                                       & {\rm if} \   l < j,  \\
  \left \langle v, Q_{n}^2 \right \rangle, & {\rm if} \   l = j,
 \end{array}
 \right.
\end{equation*}
assuming $b_{0,n}=1$ and using (\ref{2}), we derive
\begin{equation*}
\left \langle u, P_{n-j}Q_{n}\right \rangle
=  \left \langle u, P_{n-j}  \sum \limits_{i=0}^{k-1}b_{i,n}P_{n-i}  \right \rangle
= b_{j,n}  \left \langle u,  P_{n-j}^2  \right \rangle, \ \  {\rm for} \  j=0, 1 \ldots, k-1.
\end{equation*}
Now we write the equations (\ref{equations}) as a system of $k$ linear equations
$
T \bar{h} = b,
$
where
\begin{equation*}
T = \left(
\begin{array}{cccccc}
\left \langle v, Q_{n}^2 \right \rangle & \left \langle v, x  P_{n}Q_{n}\right \rangle
& \cdots & \left \langle v, x^{k-2} P_{n}Q_{n}\right \rangle & \left \langle v, x^{k-1} P_{n}Q_{n}\right \rangle \\
0 & \left \langle v, Q_{n}^2\right \rangle
& \cdots & \left \langle v, x^{k-2} P_{n-1}Q_{n}\right \rangle &  \left \langle v, x^{k-1} P_{n-1}Q_{n}\right \rangle \\
0 & 0 & \cdots &  \left \langle v, x^{k-2} P_{n-2}Q_{n}\right \rangle
 &   \left \langle v, x^{k-1} P_{n-2}Q_{n}\right \rangle   \\
\vdots & \vdots & \cdots & \vdots & \vdots \\
0 & 0 & \cdots & \left \langle v, Q_{n}^2 \right \rangle & \left \langle v, x^{k-1} P_{n-(k-2)}Q_{n}\right \rangle  \\
0 & 0 & \cdots & 0 & \left \langle v,  Q_{n}^2 \right \rangle   \\
\end{array}%
\right),
\end{equation*}
\begin{equation*}
\bar{h} = \left(
\begin{array}{c}
  h_{0}\\ h_{1} \\ h_{2} \\ \vdots \\ h_{k-2}  \\ h_{k-1}
\end{array}%
\right) \ \ \ {\rm and} \ \ \
b = \left(
\begin{array}{c}
b_{0,n}  \left \langle u,  P_{n}^2    \right \rangle \\
b_{1,n}  \left \langle u,  P_{n-1}^2  \right \rangle \\
b_{2,n}  \left \langle u,  P_{n-2}^2  \right \rangle \\
\vdots \\
b_{k-2,n}  \left \langle u,  P_{n-(k-2)}^2  \right \rangle \\
b_{k-1,n}  \left \langle u,  P_{n-(k-1)}^2  \right \rangle
\end{array}%
\right).
\end{equation*}
The latter can be rewritten in the form
\begin{equation*}
h_{j} \left \langle v,  Q^2_{n} \right \rangle + \sum_{l=j+1}^{k-1} h_{l} \left \langle v,  x^l P_{n-j}Q_{n} \right \rangle = b_{j,n}\left \langle u,  P_{n-j}^2 \right \rangle, \ \ {\rm for} \ j=0,1,\ldots,k-1.
\end{equation*}
Using the backward technique for solution of  systems of linear equations, we obtain, for $j=k-1,k-2,...,1,0$,
\begin{equation} \label{system}
h_{j} = \left. \left(
b_{j,n}\left \langle u,  P_{n-j}^2 \right \rangle - \sum_{l=j+1}^{k-1} h_{l} \left \langle v,  x^l P_{n-j}Q_{n} \right \rangle \right) \right/ \left \langle v,  Q^2_{n} \right \rangle, \ \   \mathrm{for}\ \ n \geq k.
\end{equation}

In order to simplify (\ref{system}), let $\mathbf{J}_{P}$ be the tridiagonal matrix corresponding to the SMOP $\left \{ P_{n}\right \} _{n\geq 0}$, that is,
\begin{equation*}
x\mathbf{P}=\mathbf{J}_{P}\mathbf{P},
\end{equation*}
where $\mathbf{P}= ( P_{0},P_{1} ,... ) ^{T}$ and
\begin{equation*}
\mathbf{J}_{P} = \left(\begin{array} {ccccccccc}
\beta_{0}  & 1 & 0 & 0 & \dots         & 0 & 0 & \dots\\
\gamma_{1} & \beta_{1} & 1 & 0 & \dots & 0 & 0 & \dots\\
0 & \gamma_{2} & \beta_{2} & 1 & \dots & 0 & 0 & \dots\\
\vdots & \vdots & \vdots & \vdots & \ddots & \vdots & \vdots \\
0 & 0 & 0 & 0 & \dots & \beta_{n-2}  & 1            &   \\
0 & 0 & 0 & 0 & \dots & \gamma_{n-1} & \beta_{n-1}  & \ddots \\
\vdots & \vdots  & \vdots  &  \vdots &       &              & \ddots       & \ddots
\end{array}\right).
\end{equation*}
Notice that, for $j=0,1,\ldots,k-1,$ and $l \geq j$ we have
\begin{equation*}
x^{l} P_{n-j}(x) = \sum_{i=0}^{n+l-j} (\mathbf{J}_{P}^{l})_{n-j,i} P_{i}(x),
\end{equation*}
where $(\mathbf{J}_{P}^{l})_{n-j,i}$ denotes the $(n-j, i)$  entry  of the matrix  $\mathbf{J}_{P}^{l}$.
Then the equalities
\begin{eqnarray}
\label{equ_aux_1}
 \left \langle v, x^{l} P_{n-j}Q_{n}\right \rangle & = &
 \left \langle v, \sum_{i=n}^{n+l-j} (\mathbf{J}_{P}^{l})_{n-j,i} P_{i} Q_{n} \right \rangle \\
 & = & \sum_{i=n}^{n+l-j} (\mathbf{J}_{P}^{l})_{n-j,i} \left \langle v,  P_{i} Q_{n} \right \rangle \notag
\end{eqnarray}
hold for $l \geq j$.

Now it is clear that the inner products $ \left \langle v,  P_{n+r} Q_{n} \right \rangle$,  $r= 0,1,2,\ldots, l-j$, can be expressed in terms of the coefficients $b_{i,n+i}$, $i = 1,2,\ldots, l-j$,  and from the value of $\left \langle v,  Q_{n}^2 \right \rangle$.
Indeed, we rewrite (\ref{2}) in the form
\begin{equation*}
P_{n+r}(x) = Q_{n+r}(x) - \sum \limits_{i=1}^{k-1}b_{i,n+r}P_{n+r-i}(x),
\end{equation*}
which implies
\begin{eqnarray*}
\left \langle v,  P_{n+r} Q_{n} \right \rangle
& = & \left \langle v,  \left(Q_{n+r} - \sum \limits_{i=1}^{k-1}b_{i,n+r}P_{n+r-i} \right)  Q_{n} \right \rangle  \\
& = & - \sum \limits_{i=1}^{r}b_{i,n+r} \left \langle v, P_{n+r-i}Q_{n} \right \rangle,
\end{eqnarray*}
for $r=1,2,\ldots,l-j$, so that
\begin{equation} \label{equ_aux_2}
\left \langle v,  P_{n+r} Q_{n} \right \rangle + \sum \limits_{i=1}^{r-1}b_{i,n+r} \left \langle v, P_{n+r-i}Q_{n} \right \rangle =- b_{r,n+r} \left \langle v, Q_{n}^2 \right  \rangle.
\end{equation}

 Using equations (\ref{equ_aux_2}), for $r=1,2,\ldots,l-j,$ and including the equation $ \left \langle v,  P_{n}Q_{n}   \right \rangle = \left \langle v, Q_{n}^2   \right \rangle,$ we obtain the
 following system of $(l-j+1)$ equations:
\begin{eqnarray}
&
\left(
\begin{array}{cccccccccc}
1 & 0         & 0         & 0         & \cdots & 0       & 0         \\
0 & 1         & 0         & 0         & \cdots & 0       & 0         \\
0 & b_{1,n+2} & 1         & 0         & \cdots & 0       & 0         \\
0 & b_{2,n+3} & b_{1,n+3} & 1         & \cdots & 0       & 0         \\
  \vdots   & \vdots    & \vdots    & \vdots       &    & \vdots  & \vdots  \\
0 & b_{l-j-1,n+l-j}       & b_{l-j-2,n+l-j}   & b_{l-j-3,n+l-j}    & \cdots & b_{1,n+l-j} & 1     \\
\end{array}%
\right) \times \notag
\\ &
\left(
\begin{array}{c}
\left \langle v,  P_{n}Q_{n}     \right \rangle \\
\left \langle v,  P_{n+1}Q_{n}   \right \rangle \\
\left \langle v,  P_{n+2}Q_{n}   \right \rangle \\
\vdots \\
\left \langle v,  P_{n+l-j}Q_{n}   \right \rangle
\end{array}%
\right)
=
\left(
\begin{array}{c}
1 \\
-b_{1,n+1}    \\
-b_{2,n+2}   \\
\vdots \\
-b_{l-j,n+l-j}
\end{array}%
\right)
\left \langle v,  Q_{n}^2 \right \rangle.
\label{system2}
\end{eqnarray}
Let us denote by  $A_{l-j+1}$ the matrix of the latter system.  Then the solution $ \left \langle v,  P_{n+r} Q_{n} \right \rangle$,  $r=0,1,2,\ldots, l-j$, is obtained in terms of the coefficients $b_{i,n}$, $i = 1,2,\ldots, l-j$, and  $\left \langle v,  Q_{n}^2 \right \rangle$.

 Replacing the solution of (\ref{system2}) into (\ref{equ_aux_1}) we conclude that
\begin{eqnarray*}
 \left \langle v, x^{l} P_{n-j}Q_{n}\right \rangle & = & \left(  (\mathbf{J}_{P}^{l})_{n-j,n}, (\mathbf{J}_{P}^{l})_{n-j,n+1}, \ldots, (\mathbf{J}_{P}^{l})_{n-j,n+l-j} \right) \times \\
& & A_{l-j+1}^{-1} \left(
\begin{array}{c}
1 \\
-b_{1,n+1}    \\
-b_{2,n+2}   \\
\vdots \\
-b_{l-j,n+l-j}
\end{array}%
\right) \left \langle v,  Q_{n}^2 \right \rangle,
\end{eqnarray*}
where $A_{l-j+1}^{-1} $ is the inverse of the matrix $A_{l-j+1}$.  Finally we solve the system (\ref{system}) and find all coefficients
$h_{j},$ $j=0,1,\ldots,k-1,$ of the polynomial $h$ as functions of $\beta_n, \gamma_n$ and $b_{i,n}.$
Thus, Theorem \ref{theorem_about_h} is proved.

The above result shows that  the  sequences $\{b_{j,n}\}_ {n\geq k}$,  $j= 0,1, \dots, k-1$,  defined in Theorem \ref{maintheorem} must satisfy the constraints on the coefficients of the polynomial $h(x)$ given in Theorem 2. In other words, the sequences $\{b_{j,n}\}_{n\geq k},$  $j= 0,1, \dots, k-1$, together with the coefficients of the three term recurrence relation,  determine uniquely the polynomial $h(x)$. Moreover, since the matrix $T$ is nonsingular, any polynomial $h(x)$ of the form (\ref{poly_h}) determines uniquely the coefficients $b_{0,n}, b_{1,n}, \ldots, b_{k-1,n}$, for $n\geq k$. We discuss this question thoroughly in the next section.

Notice that the latter observations provide not only an algorithm to calculate $h(x)$, but also an alternative proof about the relation between the Geronimus transformation and the quasi-orthogonal polynomials.

It is easy to see from (\ref{system}) that the leading coefficient of $h(x)$ is given, in an alternatively way, by
\begin{equation} \label{eqhk_1}
h_{k-1}  = \frac{b_{k-1,n}\left \langle u,  P_{n-(k-1)}^2 \right \rangle}{\left \langle v,  Q^2_{n} \right \rangle}, \ \ {\rm for} \ n \geq k-1.
\end{equation}
Considering $n=k-1$ and the normalization $\left \langle u,1\right \rangle =1$, we obtain
\begin{equation*}
h_{k-1}  = \frac{b_{k-1,k-1}}{\left \langle v,Q_{k-1}^{2}\right \rangle } = \frac{b_{k-1,k-1}}{\tilde{\gamma}_{1} \tilde{\gamma}_{2} \cdots \tilde{\gamma}_{k-1}\left \langle v,1\right \rangle } \neq 0,
\end{equation*}
where $\tilde{\gamma}_{1},\tilde{\gamma}_{2},...,\tilde{\gamma}_{k-1}$ are given by $\left( \ref{8}\right) $.

The second coefficient of ${h(x)}$ can also be obtained in an explicit form. Indeed, it follows from (\ref{system}) that
\begin{equation} \label{eqhk-2}
\left \langle v,Q_{n}^2 \right \rangle h_{k-2} =  b_{k-2,n} \left \langle u,   P_{n-(k-2)}^2 \right \rangle - \left \langle v, x^{k-1} P_{n-(k-2)}Q_{n}\right \rangle h_{k-1}.
\end{equation}
Now (\ref{equ_aux_1}), with $l=k-1$ and $ j= k-2$, yields
\begin{eqnarray*}
\left \langle v, x^{k-1} P_{n-(k-2)}Q_{n}\right \rangle
& = & \sum_{i=n}^{n+1} (\mathbf{J}_{P}^{k-1})_{n-(k-2),i} \left \langle v,  P_{i} Q_{n} \right \rangle \\
& = &  (\mathbf{J}_{P}^{k-1})_{n-(k-2),n} \left \langle v,  P_{n} Q_{n} \right \rangle +
 (\mathbf{J}_{P}^{k-1})_{n-(k-2),n+1} \left \langle v,  P_{n+1} Q_{n} \right \rangle \\
 & = &   \sum_{i=0}^{k-2} \beta_{n-i} \left \langle v,  Q_{n}^2 \right \rangle + \left \langle v,  P_{n+1} Q_{n} \right \rangle.
\end{eqnarray*}
Since
\begin{eqnarray*}
   Q_{n+1}(x) = P_{n+1}(x) + b_{1,n+1}  P_{n}(x) + b_{2,n+1}  P_{n-1}(x) + \cdots + b_{k-1,n+1}  P_{n-(k-2)}(x),
\end{eqnarray*}
then
$ \left \langle v,  P_{n+1} Q_{n} \right \rangle = -  b_{1,n+1}  \left \langle v,  Q_{n}^2 \right \rangle. $
Therefore (\ref{eqhk-2}) becomes
\begin{eqnarray*}
\left \langle v,Q_{n}^2 \right \rangle h_{k-2} & = &
 b_{k-2,n} \left \langle u,   P_{n-(k-2)}^2 \right \rangle - \left( \sum_{i=0}^{k-2} \beta_{n-i} - b_{1,n+1} \right) \left \langle v,  Q_{n}^2 \right \rangle h_{k-1} \\
\frac{h_{k-2}}{h_{k-1}} & = & b_{1,n+1}-\sum_{i=0}^{k-2} \beta_{n-i} + \frac{b_{k-2,n}}{h_{k-1}} \frac{\left \langle u, P_{n-(k-2)}^2 \right \rangle}{\left \langle v,  Q_{n}^2 \right \rangle}.
\end{eqnarray*}
Then (\ref{eqhk_1}) implies
\begin{eqnarray*}
\frac{h_{k-2}}{h_{k-1}} & = & b_{1,n+1}-\sum_{i=1}^{k-1} \beta_{n+1-i} + \frac{b_{k-2,n} \left \langle u, P_{n-(k-2)}^2 \right \rangle}
{b_{k-1,n}\left \langle u,  P_{n-(k-1)}^2 \right \rangle}  \\
  & = & b_{1,n+1}-\sum_{i=1}^{k-1} \beta_{n+1-i} + \frac{b_{k-2,n}}{b_{k-1,n}} \gamma_{n-k+2}.
\end{eqnarray*}

The computations of the remaining coefficients of $h(x)$ are rather involved and yield extremely complex explicit expressions so that we omit them.

\begin{rem}
Notice that the above result shows that you can find a direct relation between  the coefficients of the polynomial $h,$ the connection coefficients of the sequences $\left \{P_{n}\right \}_{n\geq0}$ and $\left \{ Q_{n}\right \}_{n\geq0}$ and the coefficients of the three term recurrence relation of the sequence $\left \{ P_{n}\right \}_{n\geq0}$.

\end{rem}

\section{Gaussian type quadrature formulas}
\label{section_matrices}

\subsection{An interpretation in terms of Jacobi matrices}

In this section we provide an alternative approach to the above problems based on the matrix form of the three-term recurrence relations as well as of the
connection coefficients between the two sequences of polynomials. Let $\mathbf{J}_{P}$ and $\mathbf{J}_{Q}$ be the tridiagonal matrices corresponding to the SMOP $\left \{ P_{n}\right \} _{n\geq 0}$ and $\left \{ Q_{n}\right \} _{n\geq 0}$, respectively. Then the three-term recurrence relations satisfied by the
SMOP $\left \{P_{n}\right \} _{n\geq 0}$ and $\left \{ Q_{n}\right \} _{n\geq 0}$ are equivalent to
\begin{equation}  \label{16}
x\mathbf{P}=\mathbf{J}_{P}\mathbf{P,\ \ \ }x\mathbf{Q}=\mathbf{J}_{Q}\mathbf{Q,}
\end{equation}%
where $\mathbf{P}= ( P_{0},P_{1} ,... ) ^{T}$ and $\mathbf{Q}=( Q_{0},Q_{1}, ...) ^{T}$.

On the other hand, $\left( \ref{2}\right) $ reads as
\begin{equation}  \label{17}
\mathbf{Q=\tilde{A}P},
\end{equation}%
where $\mathbf{\tilde{A}=}\left( \tilde{a}_{s,l}\right) _{s,l\geq 1}$ is a
banded lower triangular matrix with entries $\tilde{a}_{s,s}=1$ and $%
\tilde{a}_{s,l}=0$, $s-l>k-1$. Combining $\left( \ref{16}\right) $ and $%
\left( \ref{17}\right) $ we obtain
\begin{equation*}
x\mathbf{\tilde{A}P} =\mathbf{J}_{Q}\mathbf{\tilde{A}P}
\end{equation*}
and then
\begin{equation}
\mathbf{\tilde{A}J}_{P} =\mathbf{J}_{Q}\mathbf{\tilde{A}},\ \ \
\text{i.e.,} \ \ \ \mathbf{J}_{Q}=\mathbf{\tilde{A}J}_{P}\mathbf{\tilde{A}}^{-1}.
\label{18}
\end{equation}%
These represent a succinct matrix form of the relations obtained in Theorem \ref{maintheorem}.

On the other hand, Christoffel formula \cite{18} is equivalent to
\begin{equation}
\tilde{h}(x) \mathbf{P}=\mathbf{\tilde{B}Q}  \label{19}
\end{equation}%
where  $\mathbf{\tilde{B}}= ( \tilde{b}_{s,l})_{s,l\geq 1}$ is a
banded upper triangular matrix with entries $\tilde{b}_{s,s+k-1}=1$, $\tilde{b}_{s,l}=0$, $l-s>k-1$, and $\tilde{h}(x) = h(x)/h_{k-1} $, where $h(x)$ is the polynomial defined in (\ref{12}).

Substituting (\ref{16}) and (\ref{17}) into (\ref{19}),  we obtain
\begin{equation}
\tilde{h}\left( \mathbf{J}_{P}\right) =\mathbf{\tilde{B}\tilde{A}},
\label{20}
\end{equation}
where $\tilde{h}\left( \mathbf{J}_{P}\right) $ is a diagonal matrix of size $\left( 2k-1\right) $.  It is clear that the matrix $\mathbf{\tilde{B}}$ is uniquely determined from $\left( \ref{20}\right) $. Since equalities
(\ref{18}) and (\ref{20}) yield
\begin{equation}
\tilde{h}\left( \mathbf{J}_{Q}\right) =\mathbf{\tilde{A}}\, \tilde{h} \left( \mathbf{J}_{P}\right)\, \mathbf{\tilde{A}}^{-1}  =
\mathbf{\tilde{A}\tilde{B}},  \label{21}
\end{equation}
the matrix $\mathbf{J}_{Q}$ can be determined from (\ref{21}). Notice that (\ref{21}) is the $LU$ factorization of
the matrix $\tilde{h}\left( \mathbf{J}_{Q}\right) $ while $\left( \ref{20}\right) $
is a UL factorization of the matrix $\tilde{h}\left( \mathbf{J}_{P}\right)$.

We also describe relations between the corresponding finite dimensional tridiagonal matrices which appear in the three-term recurrence relations
(\ref{16}) as well as on (\ref{17}). If $\left(
\mathbf{P}\right) _{n}=\left \{ P_{0},P_{1},...,P_{n}\right \} ^{T}$ and $
\left( \mathbf{Q}\right) _{n}=\left \{ Q_{0},Q_{1},...,Q_{n}\right \} ^{T}$, then
(\ref{16}) and (\ref{17}) reduce to
\begin{eqnarray}
x\left( \mathbf{P}\right) _{n} &=&\left( \mathbf{J}_{P}\right) _{n+1}\left(
\mathbf{P}\right) _{n}+P_{n+1}e_{n+1},  \label{22} \\
x\left( \mathbf{Q}\right) _{n} &=&\left( \mathbf{J}_{Q}\right) _{n+1}\left(
\mathbf{Q}\right) _{n}+Q_{n+1}e_{n+1},  \label{23} \\
\left( \mathbf{Q}\right) _{n} &=& (\mathbf{\tilde{A}})_{n+1} ( \mathbf{P}) _{n} , \label{24}
\end{eqnarray}%
where $\left( .\right) _{n}$ denotes the leading principal submatrix of size $n\times n$  of the corresponding infinite one, while here and in what follows,
$e_{j}$ is the $j$-th vector of the canonical basis in $\mathbb{R}^{n+1}$ with all entries zeros except for the $j$-th one, which is
one. Replacing  (\ref{24}) and (\ref{2}) in (\ref{23})  yields
\begin{equation*}
x\ ( \mathbf{\tilde{A}})_{n+1} \left( \mathbf{P}\right) _{n} =
\left[ \left( \mathbf{J}_{Q}\right) _{n+1} ( \mathbf{\tilde{A}})_{n+1} + e_{n+1}
\left( \sum \limits_{i=1}^{k-1}b_{i,n+1}e_{n+2-i}^{T}\right)
\right] \left( \mathbf{P}\right) _{n} + P_{n+1}e_{n+1}.
\end{equation*}
Having in mind (\ref{22}), the latter simplifies to
\begin{equation*}
( \mathbf{\tilde{A}}) _{n+1}\left( \mathbf{J}_{P}\right)
_{n+1}=\left( \mathbf{J}_{Q}\right) _{n+1} ( \mathbf{\tilde{A}})_{n+1} +
(\mathbf{\tilde{A}} ) _{n+1}e_{n+1} \left( \sum
\limits_{i=1}^{k-1}b_{i,n+1}e_{n+2-i}^{T}\right).
\end{equation*}%
Thus, we obtain
\begin{equation*}
( \mathbf{J}_{Q}) _{n+1}= ( \mathbf{\tilde{A}}) _{n+1}
\left[ \left( \mathbf{J}_{P}\right) _{n+1}-e_{n+1}\left( \sum
\limits_{i=1}^{k-1}b_{i,n+1}e_{n+2-i}^{T}\right) \right] ( \mathbf{
\tilde{A}} )_{n+1}^{-1}.
\end{equation*}
This result means that $\left( \mathbf{J}_{Q}\right) _{n+1}$ is a rank-one
perturbation of the matrix $\left( \mathbf{J}_{P}\right) _{n+1}$.

\begin{rem}
The particular cases $ k=2 $, $k=3, $ and $k =4$ of the above matrix method are considered in   \cite{11}, and \cite{24}, respectively.
\end{rem}

\begin{rem}
Having in mind that the zeros of the polynomial $Q_{n+1}$ are the eigenvalues of the matrix $\left( \mathbf{J}_{Q}\right) _{n+1}$, the above expression
means that they are the eigenvalues of a rank one perturbation of the matrix  $\left( \mathbf{J}_{P}\right) _{n+1}$. Therefore, one may estimate them
using the classical theory of eigenvalue perturbations (see \cite{37}). On the other hand, the corresponding Christoffel numbers are the first component of the normalized eigenvector associated with each eigenvalue.
\end{rem}

\subsection{Results on the zeros of orthogonal polynomials}
\label{section_zeros}

In this section we discuss some properties of these zeros and of their location with respect to those
of $P_n$ provided that both $\{ P_n \}_{n\geq0}$ and $\{ Q_n \}_{n\geq0}$ are sequences of orthogonal polynomials and
they are related by (\ref{2}).

In order to obtain inequalities for the number of zeros of
$Q_n$ which are greater than the largest zero of $P_n$ we need a theorem on Descartes rule of signs
for orthogonal polynomials due to Obrechkoff. Given a finite sequence $\alpha_{0}, \ldots, \alpha_{n}$ of real numbers, let $S(\alpha_{0}, \ldots, \alpha_{n})$ be the number of its sign changes. Recall that $S(\alpha_{0}, \ldots, \alpha_{n})$
is counted in the following natural way. First we discard the zero entries from the sequence and then count a sign change if two consecutive terms in the remaining sequence have opposite signs.  By $Z(f;(a,b))$ we denote the number of the zeros, counting their multiplicities, of the function $f(x)$ in $(a,b)$.
\begin{dfn}
The sequence of functions $f_{0},\ldots, f_{n}$ obeys the general
Descartes' rule of signs in the interval $(a,b)$ if the number of
zeros in $(a,b)$, where the multiple zeros are counted with their
multiplicities, of any real nonzero linear combination
\begin{equation*}
\alpha_{0} f_{0}(x) + \ldots + \alpha_{n} f_{n}(x)
\end{equation*}
does not exceed the number of sign changes in the sequence
$\alpha_{0}, \ldots, \alpha_{n}$.
\end{dfn}

More precisely, this property states that
\begin{equation*}
Z(\alpha_{0} f_{0}(x) + \ldots + \alpha_{n} f_{n}(x);(a,b)) \leq S(\alpha_{0}, \ldots, \alpha_{n})
\end{equation*}
for any $(\alpha_{0}, \ldots, \alpha_{n})\neq (0,\ldots,0)$.

\begin{THEO}[Obrechkoff~\cite{Obr}]\label{thO}
If the sequence of polynomials $\{p_{n}\}_{n\geq0}$ is defined by the recurrence
relation
\begin{equation*} 
x p_{n}(x) = a_{n} p_{n+1}(x) + b_{n} p_{n}(x) + c_{n} p_{n-1}(x), \ \
n \geq 0, 
\end{equation*}
with $p_{-1}(x) = 0$ and \ $p_{0}(x) = 1$,
where $a_{n}, b_{n}, c_{n} \in \mathbb{R}$, $\ a_{n}, c_{n}
> 0$ and $z_{n}$ denotes the largest zero of $p_{n}(x)$, then
the sequence of polynomials $p_{0}, \ldots, p_{n}$ obeys
Descartes' rule of signs in $(z_{n},\infty)$.
\end{THEO}

Since, by Favard's theorem \cite{Fav35}, the
requirements on $p_k(x)$ in Theorem \ref{thO} are equivalent to
the fact that $\{p_{n}\}_{n\geq0}$ is a sequence of orthogonal
polynomials, we obtain

\begin{cor}\label{corO}
Suppose that the orthogonal polynomials $p_k(x)$, $k=0,1,\ldots, n$ be
normalized in such a way that their leading coefficients are all
of the same sign and let $z_n$ be the largest zero of
$p_n(x)$. Then, for any set of real numbers $\alpha_{0}, \ldots, \alpha_{n}$,
which are not identically zero, we get
\begin{equation*}
Z(\alpha_{0} p_0(x) + \cdots + \alpha_{n} p_n(x);(z_n,\infty)
) \leq S(\alpha_{0}, \ldots, \alpha_{n}).
\end{equation*}
\end{cor}

Some applications of Theorem \ref{thO} and Corollary \ref{corO} to zeros of orthogonal polynomials were discussed in \cite{Dim01}

Now we are ready to formulate a result concerning inequalities for largest zeros of the polynomials $Q_n$.

\begin{thm}  
Let $\{ P_{n} \}_{n\geq0}$ be a sequence of monic orthogonal polynomials
and let $\{ Q_{n} \}_{n\geq k}$ be defined by (\ref{2}). If the zeros of $P_n(x)$ are
$x_{n,1}< \cdots < x_{n,n}$, then
\begin{equation*}
Z(Q_n(x),(x_{n,n},\infty)) \leq S(1,b_{1,n},\ldots,b_{k-1,n}).
\end{equation*}
\end{thm}

Despite that in this paper we are interested in the situation when $\{ Q_n \}_{n\geq 0}$ is another sequence of orthogonal polynomials, the above result about the largest zeros of $Q_n$ does not depend on the fact that
the sequence of polynomials obeys an orthogonality property or not.

\begin{cor}
If the zeros of $Q_n$ are also real and simple, denoted by
$y_{n,1}< \cdots < y_{n,n}$ and $S(1,b_{1,n},\ldots,b_{k-1,n})=\ell$, then
\begin{equation*}
y_{n,n-\ell} < x_{n,n}.
\end{equation*}
In particular $y_{n,n-k+1} < x_{n,n}$ independently of the signs of $b_{i,n}\geq 0$ for $i=1,\ldots,n-k+1$.
Moreover, if $b_{i,n}\geq 0$ for $i=1,\ldots,n-k+1$, then $y_{n,n} < x_{n,n}$ which means that all zeros of $Q_n$ precede $x_{n,n}$.
\end{cor}

Finally, we obtain a relation between the Stieltjes functions of $u$ and $v.$ Indeed, let define
\begin{equation*}
S_{u}(z) = \sum_{n=0}^{\infty} \frac{u_n}{z^{n+1}} \qquad {\rm and}
\qquad
S_{v}(z) = \sum_{n=0}^{\infty} \frac{v_n}{z^{n+1}},
\end{equation*}
where $u_{n}=\left \langle u,x^{n}\right \rangle $ and
$v_{n}=\left \langle v,x^{n}\right \rangle $.

Since  $\left \langle u,x^{n}\right \rangle  = \left \langle v,h(x) x^{n}\right \rangle$ then $\displaystyle u_{n} = \sum_{j=0}^{k-1} h_{j} v_{j+n}$, \ and
\begin{eqnarray*}
S_{u}(z) & = & \sum_{n=0}^{\infty} \frac{1}{z^{n+1}}  \left( \sum_{j=0}^{k-1} h_{j} v_{j+n}\right) =
\sum_{j=0}^{k-1} h_{j} z^{j}  \left( \sum_{n=0}^{\infty}  \frac{v_{j+n}}{z^{j+n+1}}  \right) \\
& = & \sum_{j=0}^{k-1} h_{j} z^{j}  \left( S_{v}(z) - \sum_{s=0}^{j-1}  \frac{v_{s}}{z^{s+1}}  \right) \\
& = &
\sum_{j=0}^{k-1} h_{j} z^{j} S_{v}(z) - \sum_{j=0}^{k-1} h_{j} z^{j} \left(\sum_{s=0}^{j-1}  \frac{v_{s}}{z^{s+1}}  \right).
\end{eqnarray*}
Therefore,
$ S_{u}(z)= h(z)S_{v}(z) - T(z), $
where $ \displaystyle T(z) = \sum_{j=0}^{k-1} h_{j} z^{j} \left(\sum_{s=0}^{j-1}  \frac{v_{s}}{z^{s+1}}   \right)$ is a polynomial of degree at most $k-2$, and
\begin{equation*}
S_{v}(z)= \frac{S_{u}(z)}{h(z)} + \frac{T(z)}{h(z)}.
\end{equation*}

Since the Stieltjes function $S_{v}$ is a linear spectral modification of $S_{u}$ (\cite{38}), assuming that $u$ is a positive definite linear functional and $h$ is a positive polynomial on the support of a positive Borel measure $d\mu$  associated with $u$, it is well known (see \cite{Lago89} and \cite{Lago90}) that for $n$ large enough each zero $\zeta$ of $h$  with multiplicity $j$ attracts $j$ zeros of $Q_n$. On the other hand, for every fixed $n$, at most $k-1$ zeros of $Q_{n}$ can lie outside $supp(\mu)$. These facts allow us to judge about the location of the zeros of $h$ that lie outside the support of $d\mu$ .

\subsection{Kernel polynomials and Christoffel numbers}

In \cite{BCBvB} quadrature formulas on the real line with the highest degree of accuracy, with positive weights, and with one or two prescribed nodes anywhere on the interval of integration are characterized. Next we will consider a more general problem when we deal with more prescribed nodes. We are interested in the study of Christoffel numbers assuming they are positive numbers, i.e. by choosing  those nodes outside the interval of orthogonality of the initial measure.

Let $K_{n}(x,y;u)$ and  $K_{n}(x,y;v)$ be the kernel polynomials associated with the positive definite linear functionals $u$ and $v$, respectively, i.e.
\begin{equation*}
K_{n}(x,y;u) = \sum_{j=0}^{n} \frac{P_{j}(x) P_{j}(y)}{||P_{j}||^2}
\quad \mbox{and} \quad
K_{n}(x,y;v) = \sum_{j=0}^{n} \frac{Q_{j}(x) Q_{j}(y)}{||Q_{j}||^2},
\end{equation*}
where $||P_{m}||^2 =\left\langle u, P_{m}(x) P_{m}(x) \right \rangle$ \
and \ $||Q_{m}||^2 =\left\langle v, Q_{m}(x) Q_{m}(x) \right \rangle.$ \\

First of all, we will find an algebraic relation between  $K_{n}(x,y;u)$ and  $K_{n}(x,y;v)$.

Writing $K_{n}(x,y;v) $ as
\begin{equation*}
K_{n}(x,y;v) = \sum_{m=0}^{n} \alpha_{n,m}(y) P_{m}(x),
\end{equation*}
we get
\begin{equation*}
\alpha_{n,m}(y) = \frac{\left\langle u, K_{n}(x,y;v) P_{m}(x) \right \rangle}{||P_{m}||^2}.
\end{equation*}

If $m \leq n-k+1,$ then
$ \left\langle u, K_{n}(x,y;v) P_{m}(x) \right \rangle =
\left\langle v, K_{n}(x,y;v) h(x) P_{m}(x) \right \rangle. $
From the reproducing property of the kernel polynomial we get
\begin{equation*}
\alpha_{n,m}(y) = \frac{h(y) P_{m}(y)}{||P_{m}||^2}, \quad for \quad 0 \leq m \leq n-k+1.
\end{equation*}

On the other hand,
\begin{eqnarray*}
\alpha_{n,n-k+2}(y)
& = &  \frac{\left\langle u, K_{n}(x,y;v) P_{n-k+2}(x) \right \rangle}{||P_{n-k+2}||^2} \\
& = &  \frac{\left\langle v, \left[ K_{n+1}(x,y;v) - \frac{Q_{n+1}(x)Q_{n+1}(y)}{||Q_{n+1}||^2} \right] h(x) P_{n-k+2}(y) \right \rangle}{||P_{n-k+2}||^2} \\
& = &  \frac{h(y)P_{n-k+2}(y)}{||P_{n-k+2}||^2} -\frac{Q_{n+1}(y)}{||Q_{n+1}||^2} b_{k-1,n+1}.
\end{eqnarray*}

\begin{eqnarray*}
\alpha_{n,n-k+3}(y)
& = & \frac{\left\langle u, K_{n}(x,y;v) P_{n-k+3}(x) \right \rangle}{||P_{n-k+3}||^2} \\
& = & \frac{\left\langle v, \left[ K_{n+2}(x,y;v) - \frac{Q_{n+2}(x)Q_{n+2}(y)}{||Q_{n+2}||^2} - \frac{Q_{n+1}(x)Q_{n+1}(y)}{||Q_{n+1}||^2} \right] h(x) P_{n-k+3}(x) \right \rangle}{||P_{n-k+3}||^2} \\
& = & \frac{h(y)P_{n-k+3}(y)}{||P_{n-k+3}||^2} - \frac{Q_{n+2}(y)}{||Q_{n+2}||^2} b_{k-1,n+2}  -  \frac{Q_{n+1}(y)}{||Q_{n+1}||^2} b_{k-2,n+1}.
\end{eqnarray*}

Finally,
\begin{eqnarray*}
\alpha_{n,n}(y)
& = & \frac{\left\langle u, K_{n}(x,y;v) P_{n}(x) \right \rangle}{||P_{n}||^2} \\
& = & \frac{\left\langle v, \left[ \displaystyle K_{n+k-1}(x,y;v) - \sum_{j=n+1}^{n+k-1} \frac{Q_{j}(x)Q_{j}(y)}{||Q_{j}||^2} \right] h(x) P_{n}(x) \right \rangle}{||P_{n}||^2} \\
& = & \frac{h(y)P_{n}(y)}{||P_{n}||^2} - \sum_{j=n+1}^{n+k-1} \frac{Q_{j}(y)}{||Q_{j}||^2} b_{j-n,j}.
\end{eqnarray*}

In other words,
\begin{equation*}
K_{n}(x,y;v) =  h(y) K_{n}(x,y;u) - [\mathbb{P}_{n-k+2}^{(k-1)}(x)]^T \mathbb{T}_{n,k-1} \mathbb{D}_{k-1} \mathbb{Q}_{n+1}^{(k-1)}(y),
\end{equation*}
where
\begin{equation*}
\mathbb{P}_{n-k+2}^{(k-1)}(x) = \left( P_{n-k+2}(x),  P_{n-k+3}(x), \ldots,  P_{n}(x) \right)^T,
\end{equation*}
\begin{equation*}
\mathbb{Q}_{n+1}^{(k-1)}(y) = \left( Q_{n+1}(y),  Q_{n+2}(y), \ldots,  Q_{n+k-1}(y) \right)^T,
\end{equation*}

\begin{equation*}
 \mathbb{T}_{n,k-1} =
 \left(
\begin{array}{cccccccccc}
 b_{k-1,n+1} & 0           & 0           & \cdots  & 0         \\
 b_{k-2,n+1} & b_{k-1,n+2} & 0           & \cdots  & 0         \\
 b_{k-3,n+1} & b_{k-2,n+2} & b_{k-1,n+3} & \cdots  & 0         \\
 \vdots      & \vdots      & \vdots      &         & \vdots    \\
 b_{1,n+1}   & b_{2,n+2}   & b_{3,n+3}   & \cdots  &  b_{k-1,n+k-1}    \\
\end{array}
\right)
\end{equation*}
and
\begin{equation*}
\mathbb{D}_{k-1}  = diag\left( \frac{1}{||Q_{n+1}||^2}, \frac{1}{||Q_{n+2}||^2}, \ldots,  \frac{1}{||Q_{n+k-1}||^2} \right).
\end{equation*}

\medskip

By setting $ \mathbb{L}_{n,k-1} = \mathbb{T}_{n,k-1} \mathbb{D}_{k-1} $
we get
\begin{equation} \label{kernelv}
K_{n}(x,y;v) =  h(y) K_{n}(x,y;u) - [\mathbb{P}_{n-k+2}^{(k-1)}(x)]^T \mathbb{L}_{n,k-1} \mathbb{Q}_{n+1}^{(k-1)}(y).
\end{equation}

If we commute the variables in (\ref{kernelv}),
\begin{equation} \label{kernelvcommuted}
K_{n}(y,x;v) =  h(x) K_{n}(y,x;u) - [\mathbb{P}_{n-k+2}^{(k-1)}(y)]^T \mathbb{L}_{n,k-1} \mathbb{Q}_{n+1}^{(k-1)}(x),
\end{equation}
since the kernel polynomials are symmetric with respect to the variables, then subtracting (\ref{kernelvcommuted}) from  (\ref{kernelv}), we get
\begin{equation} \label{kernel333}
K_{n}(x,y;u) = \frac{ [\mathbb{P}_{n-k+2}^{(k-1)}(y)]^T \mathbb{L}_{n,k-1} \mathbb{Q}_{n+1}^{(k-1)}(x) -[\mathbb{P}_{n-k+2}^{(k-1)}(x)]^T \mathbb{L}_{n,k-1} \mathbb{Q}_{n+1}^{(k-1)}(y)}{h(x)-h(y)}.
\end{equation}

Substituting (\ref{kernel333}) in (\ref{kernelv}) we obtain
\begin{equation} \label{kernel444}
K_{n}(x,y;v) = \frac{ h(y)[\mathbb{P}_{n-k+2}^{(k-1)}(y)]^T \mathbb{L}_{n,k-1} \mathbb{Q}_{n+1}^{(k-1)}(x) - h(x)[\mathbb{P}_{n-k+2}^{(k-1)}(x)]^T \mathbb{L}_{n,k-1} \mathbb{Q}_{n+1}^{(k-1)}(y)}{h(x)-h(y)}.
\end{equation}

In particular, the confluent formula holds
{\small
\begin{equation*}
K_{n}(x,x;v) = \frac{ [h(x)[\mathbb{P}_{n-k+2}^{(k-1)}(x)]^T]^\prime\mathbb{L}_{n,k-1} \mathbb{Q}_{n+1}^{(k-1)}(x) - h(x)[\mathbb{P}_{n-k+2}^{(k-1)}(x)]^T \mathbb{L}_{n,k-1} [\mathbb{Q}_{n+1}^{(k-1)}(x)]^\prime}{-h^\prime(x)}.
\end{equation*}
}
or, alternatively from (\ref{kernelv})
\begin{equation*}
K_{n}(x,x;v) = h(x) K_{n}(x,x;u) - [\mathbb{P}_{n-k+2}^{(k-1)}(x)]^T \mathbb{L}_{n,k-1} \mathbb{Q}_{n+1}^{(k-1)}(x).
\end{equation*}

On the other hand, from (\ref{kernelv}) and taking into account that
\begin{equation*}
 [\mathbb{P}_{n-k+2}^{(k-1)}(x)]^T \mathbb{T}_{n,k-1} =
[\mathbb{Q}_{n+1}^{(k-1)}(x)]^T
- [\mathbb{P}_{n+1}^{(k-1)}(x)]^T
\mathbb{Z}_{n,k-1}
\end{equation*}
where
\begin{equation*}
\mathbb{Z}_{n,k-1} = \left( \begin{array}{cccccccccc}
1  & b_{1,n+2}    & b_{2,n+3}      & \cdots   & b_{k-2,n+k-1}         \\
0  & 1         & b_{1,n+3}         & \cdots   & b_{k-3,n+k-1} \\
0  & 0         & 1                 & \cdots   & b_{k-4,n+k-1} \\
\vdots    & \vdots       & \vdots  &              & \vdots    \\
0  & 0         & 0   & \cdots   & b_{1,n+k-1} \\
0  & 0         & 0   & \cdots   & 1    \\
\end{array} \right),
\end{equation*}
we get 
\begin{equation*}
K_{n+k-1}(x,y;v) =  h(y) K_{n}(x,y;u) +
[\mathbb{P}_{n-k+2}^{(k-1)}(x)]^T  \mathbb{Z}_{n,k-1}
 \mathbb{D}_{k-1} \mathbb{Q}_{n+1}^{(k-1)}(y),
\end{equation*}
and using the same arguments as above to obtain formula (\ref{kernel444}), we get the following compact expression for the kernel polynomial.

\begin{prop}
{\small 
\begin{equation*}
K_{n+k-1}(x,y;v) = \frac{ h(x)[\mathbb{P}_{n+1}^{(k-1)}(x)]^T \mathbb{M}_{n,k-1}    \mathbb{Q}_{n+1}^{(k-1)}(y) - h(y)[\mathbb{P}_{n+1}^{(k-1)}(y)]^T \mathbb{M}_{n,k-1}    \mathbb{Q}_{n+1}^{(k-1)}(x) }{h(x)-h(y)},
\end{equation*}
}
where $  \mathbb{M}_{n,k-1}   = \mathbb{Z}_{n,k-1}  \mathbb{D}_{k-1}$.
\end{prop}

\begin{rem}
Proceeding as above one has the expression for the confluent formula $ K_{n+k-1}(x,x;v)$.
\end{rem}

\begin{rem}
If $h(x)=x-a,$ then  $\mathbb{Z}_{n,1}  =1$. Thus
\begin{equation*}
K_{n+1}(x,y;v) = \frac{ (x-a) P_{n+1}(x) Q_{n+1}(y)- (y-a) P_{n+1}(y) Q_{n+1}(x) }{(x-y)||Q_{n+1}||^2 }.
\end{equation*}
If we denote by $y_{n+1,j}, j=1,2,\ldots, n+1,$ the zeros of the polynomial $Q_{n+1}$, we deduce in a straightforward way the value of the Christoffel numbers in the quadrature formula by using the above zeros as nodes.  Indeed,
\begin{equation*}
\frac{1}{K_{n+1}(y_{n+1,j},y_{n+1,j};v)} = \frac{1}{b_{1,n+1} (y_{n+1,j}-a) P_{n}(y_{n+1,j}) Q^{\prime}_{n+1}(y_{n+1,j})}.
\end{equation*}
\end{rem}

\section{Examples}

In this section we analyze some examples which illustrate the problems considered in the previous sections. First  we focus our attention on the symmetric case which is less complex than the general one. The case when the connection coefficients are constant real numbers is also studied.

\subsection{Symmetric case}

Let us consider the symmetric SMOP $\{P_{n}\}_{n\geq0}$, that is the case when  $\beta_{n}=0$ for $n\geq0$.
According to Theorem \ref{maintheorem}, equations (\ref{eq_for_b1}), (\ref{eq_for_b2}) and (\ref{eq_for_b_i}) become
\begin{equation} \label{eq_for_b1_sym}
b_{1,n+1}  =  b_{1,n} +  \frac{b_{k-2,n-1}}{b_{k-1,n-1}} \gamma_{n-k+1} - \frac{b_{k-2,n}}{b_{k-1,n}} \gamma_{n-k+2}, \quad n \geq k,
\end{equation}
\begin{equation} \label{eq_for_b2_sym}
b_{2,n+1} = b_{2,n} + \gamma_{n} -  \frac{b_{k-1,n}}{b_{k-1,n-1}}\gamma_{n-k+1} + b_{1,n}\left(b_{1,n+1}-b_{1,n}\right),  \  n \geq k,
\end{equation}
and for $1 \leq i \leq k-3$
\begin{eqnarray} 
b_{i+2,n+1} & = &   b_{i+2,n}  + b_{i+1,n}\left(b_{1,n+1} -b_{1,n}\right) + b_{i,n}\gamma_{n-i}  \notag\\
&   & - b_{i,n-1}\left[\gamma_{n} + b_{2,n} - b_{2,n+1} + b_{1,n} \left(b_{1,n+1}-b_{1,n}\right)  \right].  \label{eq_for_b_i_sym}
\end{eqnarray}
Equations (\ref{7}) and (\ref{8}) become
\begin{eqnarray*}
\tilde{\beta}_{n}  &=&  b_{1,n}-b_{1,n+1}, \ \ \ n\geq 0, \\
\tilde{\gamma}_{n} &=&  \gamma _{n}+b_{2,n}-b_{2,n+1}+b_{1,n}\left(b_{1,n+1}-b_{1,n}\right), \ \ \ n\geq 1,
\end{eqnarray*}
or alternatively
\begin{eqnarray}
\tilde{\beta}_{n}  &=&  \frac{b_{k-2,n}}{b_{k-1,n}} \gamma_{n-k+2} -
\frac{b_{k-2,n-1}}{b_{k-1,n-1}} \gamma_{n-k+1}, \quad n \geq k,
\notag \\
\tilde{\gamma}_{n} &=& \frac{b_{k-1,n}}{b_{k-1,n-1}} \gamma_{n-k+1} \quad n \geq k. \label{sym_gamma}
\end{eqnarray}

\noindent \textbf{Step 1. } If we fix $b_{1,n} =  b_{1}$ for $n \geq k$, then from (\ref{eq_for_b1_sym})
\begin{eqnarray*}
  \frac{b_{k-2,n}}{b_{k-1,n}}\gamma _{n-k+2} = \frac{b_{k-2,n-1}}{b_{k-1,n-1}}\gamma _{n-k+1} = \cdots =
 \frac{b_{k-2,k-1}}{b_{k-1,k-1}}\gamma _{1},
\end{eqnarray*}
and it is easy to conclude that $\tilde{\beta}_{n}  = 0,$ for $n\geq k.$

Relation (\ref{eq_for_b2_sym}) yields
\begin{equation} \label{b2_sym_const}
b_{2,n+1}  = b_{2,n} +\gamma _{n} -  \frac{b_{k-1,n}}{b_{k-1,n-1}}\gamma _{n-k+1}.
\end{equation}

\noindent \textbf{Step 2. } If we impose the restrictions $b_{1,n} = b_{1}$ and $b_{2,n} = b_{2}$, for $n \geq k$, then from (\ref{b2_sym_const}) and (\ref{sym_gamma}), we obtain
\begin{equation*}
\gamma_{n} = \frac{b_{k-1,n}}{b_{k-1,n-1}}\gamma _{n-k+1} = \tilde{\gamma}_{n}, \ \ {\rm for} \ \ n \geq k.
\end{equation*}

\begin{prop}
If $b_{1,n} = b_{1}$ and $b_{2,n} = b_{2}$, for $n \geq k$, then
\begin{eqnarray*}
\tilde{\beta}_{n}  &=& 0, \ \ \ n\geq k, \\
\tilde{\gamma}_{n} &=&  \gamma_{n}, \ \ \ n\geq k.
\end{eqnarray*}
This means that $Q_{n}^{[k+1]}(x) = P_{n}^{[k+1]}(x)$.
\end{prop}

Here, for a fixed positive integer number $s$, we denote by  $\{ P^{[s]}_{n}\}_{n\geq0}$ the sequence of polynomials  satisfying the three-term recurrence relation 
\begin{equation*}
x  P^{[s]}_{n}\left (x \right) =  P^{[s]}_{n+1}(x) +\beta _{n+s} P^{[s]}_{n}\left(
x\right) +\gamma _{n+s}  P^{[s]}_{n-1}(x), \ \ n\geq 0,
\end{equation*}
with  initial conditions $P^{[s]}_{-1}(x) =0$, $ P^{[s]}_{0}\left(
x\right) =1.$ It  is said to be the sequence of associated monic polynomials of order $s$ for the linear functional $u$ (see \cite{14}).  \\

\noindent \textbf{Step 3. }  We keep $b_{1,n} = b_{1}$ and $b_{2,n} = b_{2}$, for $n \geq k$, and we add the constrain $b_{3,n} = b_{3}$, for $n \geq k$. Since from (\ref{eq_for_b_i_sym}), with $i=1$,
\begin{eqnarray*}
b_{3,n+1}  & = &  b_{3,n} +b_{1} \left(\gamma_{n-1} -  \gamma _{n}\right), \ \ n \geq k+1,
\end{eqnarray*}
then
$
b_{1}(\gamma_{n-1} - \gamma_{n})=0.
$
Thus, either $b_{1} =0$ or $\gamma_{n}$ remains constant for $n \geq k$, that is,  $\gamma_{n}=\gamma_{k}$ for $n \geq k.$

If the coefficients $\gamma_{n}$ are constants for $ n \geq k$,  then $\{P_{n}\}_{n\geq0}$  is the sequence of anti-associated polynomials of order $k$ for the  Chebyshev polynomials of the second kind (see \cite{RonAss}). \\

\noindent \textbf{Step 4. } The other possibility is that $b_{1} = 0$, $b_{2,n} = b_{2}$ and $b_{3,n} = b_{3}$, for $n \geq k+1$. Now we add the restriction $b_{4,n} = b_{4}$, for $n \geq k+1$. Since, from (\ref{eq_for_b_i_sym}) with $i=2$,
\begin{eqnarray*}
b_{4,n+1} & = &   b_{4,n} +b_{2}(\gamma_{n-2} - \gamma_{n}), \ \ n \geq k+1,
\end{eqnarray*}
we obtain
\begin{eqnarray*}
b_{2}(\gamma_{n-2} - \gamma_{n})=0, \ \ n \geq k+1,
\end{eqnarray*}
and, again,  either $b_{2} =0$  or the sequence $\{\gamma _{n}\}_{n\geq k-1}$ is a periodic sequence with period 2. Thus  $\{P_{n}\}_{n\geq0}$ is the sequence of anti-associated polynomials of order $k-1$ of a 2-periodic sequence (see \cite{RonAss}).
We refer to  \cite[p.91]{14} for the explicit expression of symmetric orthogonal polynomials defined by recurrence relations whose coefficients are  2-periodic sequences.
Let  $S_{n}(x) = xS_{n-1}(x)- \gamma_{n}S_{n-2}(x)$, where $\gamma_{2n} = a >0$ and $\gamma_{2n+1} = b >0$. Then
\begin{eqnarray*}
S_{2n}(x) & = & (ab)^{n/2} \Big[ U_{n}(z) + \sqrt{b/a} \, U_{n-1}(z) \Big], \\
S_{2n+1}(x) & = & (ab)^{n/2} x U_{n}(z),
\end{eqnarray*}
where $z = (x^2 - (a+b) ) / (4ab)^{1/2}$. \\

\noindent \textbf{Step 5. } Yet another possibility is $b_{1} = 0$, $b_{2} = 0$, $b_{3,n} = b_{3}$ and $b_{4,n} = b_{4}$, for $n \geq k+1$. Following the previous reasoning let to add the restriction $b_{5,n} = b_{5}$, for $n \geq k+1$. Then (\ref{eq_for_b_i_sym}), for $i=3$, reads
\begin{eqnarray*}
b_{5,n+1}  & = &   b_{5,n} + b_{3}(\gamma_{n-3} - \gamma_{n}), \ \ n \geq k+1.
\end{eqnarray*}
Hence,
\begin{eqnarray*}
b_{3}(\gamma_{n-3} - \gamma_{n})=0, \ \ n \geq k+1.
\end{eqnarray*}
Then either $b_{3}=0$ or the sequence $\{\gamma _{n}\}_{n\geq k-1}$ is a 3-periodic one.

We can proceed in this way up to  $i=k-3$ using (\ref{eq_for_b_i_sym}), and periodic sequences appear in a natural way.

We will illustrate the above method in the case of Chebyshev polynomials of the second  kind.

\begin{exm}
\label{example_Un}
Let  $\{P_{n}\}_{n\geq0}$ be the sequence of monic Chebyshev polynomials of second kind $\{\tilde{U}_{n}\}_{n\geq0}$ orthogonal with respect to $d \mu(x) = (1-x^2)^{1/2}dx$ on $(-1,1)$. Then  $\beta_{n}=0$, $\gamma_n= 1/4$, $n \geq 1$ and
$(\ref{eq_for_b1_sym})$, $(\ref{eq_for_b2_sym})$ and $(\ref{eq_for_b_i_sym})$ become
\begin{eqnarray*}
b_{1,n+1} &=& b_{1,n} +  \frac{1}{4}\left(\frac{b_{k-2,n-1}}{b_{k-1,n-1}} - \frac{b_{k-2,n}}{b_{k-1,n}}\right), \ \ n \geq k \\
b_{2,n+1} &=& b_{2,n} +  \frac{1}{4}\left(1 -  \frac{b_{k-1,n}}{b_{k-1,n-1}} \right)
 + b_{1,n}\left(b_{1,n+1}-b_{1,n}\right),  \  n \geq k, \\
b_{i+2,n+1} & = &   b_{i+2,n} + \frac{1}{4}b_{i,n} + b_{i+1,n}\left(b_{1,n+1}-b_{1,n}\right)  \notag\\
&   & - b_{i,n-1}\left[\frac{1}{4} + b_{2,n} - b_{2,n+1} + b_{1,n} \left(b_{1,n+1}-b_{1,n}\right)  \right],
\end{eqnarray*}
for $1 \leq i \leq k-3$.

Assume that  $b_{1,n}=b_{1}$ for $n \geq k$, and $b_{2,n}=b_{2}$ for $n \geq k$. Then we have
\begin{equation*}
b_{i+2,n+1} = b_{i+2,n} + \frac{1}{4} \left(b_{i,n} - b_{i,n-1}  \right), \ \ 1 \leq i \leq k-3, \ \ n \geq k.
\end{equation*}
In particular,  according to the fact that  $b_{1,n}=b_{1}$ and $b_{2,n}=b_{2}$ for $n \geq k$, then
\begin{eqnarray*}
b_{3,n+1} & = & b_{3,n}, \ \ n \geq k+1, \\
b_{4,n+1} & = & b_{4,n}, \ \ n \geq k+1,
\end{eqnarray*}
and, as a consequence, for every $1 \leq i \leq k-3$,
\begin{eqnarray*}
b_{i+2,n+1} & = & b_{i+2,n}, \ \ n \geq k+1.
\end{eqnarray*}

On the other hand, if you assume, instead of $b_{1,n}=b_{1}$ and $b_{2,n}=b_{2}$ for $n \geq k$, that
$b_{k-1,n}=b_{k-1}$ and $b_{k-2,n}=b_{k-2}$ for $n \geq k$, a reverse situation in terms of the connection coefficients, then
 \begin{eqnarray*}
b_{1,n+1} & = & b_{1,n}, \ \ n \geq k+1, \\
b_{2,n+1} & = & b_{2,n}, \ \ n \geq k+1,
\end{eqnarray*}
and
\begin{equation*}
b_{i+2,n+1} = b_{i+2,n} + \frac{1}{4} \left(b_{i,n} - b_{i,n-1}  \right), \ \ 1 \leq i \leq k-3, \ \ n \geq k+1,
\end{equation*}
In particular, this means that
\begin{eqnarray*}
b_{i+2,n+1} & = & b_{i+2},  \ \ 1 \leq i \leq k-5, \ \ n \geq k+1.
\end{eqnarray*}
Notice that in this case
\begin{eqnarray*}
b_{1,k+1} &=& b_{1,k} + \frac{1}{4}\left(\frac{b_{k-2,k-1}}{b_{k-1,k-1}} - \frac{b_{k-2,k}}{b_{k-1,k}}\right),   \\
b_{2,k+1} &=& b_{2,k} + \frac{1}{4}\left(1 -  \frac{b_{k-1,k}}{b_{k-1,k-1}} \right),     \\
b_{i+2,k+1} &=& b_{i+2,k} + \frac{1}{4} \left(b_{i,k} - b_{i,k-1}  \right),  \ \ 1 \leq i \leq k-3.
\end{eqnarray*}

In other words, we have constant connection coefficients, but they appear for $n \geq k+1$.
\end{exm}

\begin{prop}
Let assume that $\{Q_{n}\}_{n\geq0}$ is a sequence of quasi-orthogonal polynomials of order $k-1$ with respect to the sequence $\{\tilde{U}_n\}_{n\geq0}$. If either $b_{1,n}=b_{1}$ and $b_{2,n}=b_{2}$ for $n \geq k$, or  $b_{k-1,n}=b_{k-1}$ and $b_{k-2,n}=b_{k-2}$ for $n \geq k$, then all the remaining connection coefficients are constant for $n \geq k+1$. Notice that if the initial conditions are  $b_{k-1,k}=b_{k-1,k-1}$ and $b_{k-2,k}=b_{k-2,k-1}$ then  all coefficients are constant for $n \geq k$.
In this case,
\begin{eqnarray*}
\tilde{\beta}_{n}  &=& 0, \ \ \ n\geq k, \\
\tilde{\gamma}_{n} &=& \frac{1}{4}, \ \ \ n\geq k+1.
\end{eqnarray*}
\end{prop}
This means that the SMOP $\{Q_{n}\}_{n\geq0}$  has the same sequence of ($k+1$)-associated polynomials that the SMOP $\{\tilde{U}_{n}\}_{n\geq0}$. In other words it  is an anti-associated SMOP of order $k+1$ of the Chebyshev polynomials of second kind.

\subsection{Non-symmetric case}

Notice that key information for the sequence $\{Q_{n}\}_{n\geq0}$ is given by the sequences  $\{b_{1,n}\}_{n\geq0}$ and $\{b_{2,n}\}_{n\geq0}$
or, alternatively, by the sequences $\{b_{k-2,n}\}_{n\geq k-1}$ and $\{b_{k-1,n}\}_{n\geq k-1}$    because
\begin{eqnarray*}
\tilde{\beta}_{n}  &=&  \beta _{n} +b_{1,n}-b_{1,n+1}, \ \  n\geq 0, \\
\tilde{\gamma}_{n} &=&  \gamma _{n}+b_{2,n}-b_{2,n+1}+b_{1,n}\left( \beta_{n-1}-\beta _{n}-b_{1,n}+b_{1,n+1}\right), \ \ n\geq 1, \\
\tilde{\gamma}_{n} &=&  \gamma _{n-k+1} \frac{b_{k-1,n}}{b_{k-1,n-1}}, \ \ n \geq k.
\end{eqnarray*}

If for $n \geq k$ the coefficients $b_{1,n}$ and $b_{2,n}$ do not depend on $n$,  i.e. $b_{1,n}=b_{1}$ and $b_{2,n}=b_{2}$, we have
\begin{eqnarray*}
\tilde{\beta}_{n}  &=&\beta _{n}, \ \ \ n \geq k, \\
\tilde{\gamma}_{n} &=&\gamma _{n}+b_{1}\left( \beta_{n-1}-\beta _{n}\right), \ \ \ n \geq k.
\end{eqnarray*}
On the other hand, if $b_{k-1,n}$ and $b_{k-2,n}$ are constant coefficients, for $n \geq k$, i.e. $b_{k-1,n}=b_{k-1}$ and $b_{k-2,n}=b_{k-2}$, it follows from (\ref{7}), (\ref{5}) and (\ref{6}) that
\begin{eqnarray*}
\tilde{\beta}_{n}  &=&  \beta _{n-k+1} +\frac{b_{k-2}}{b_{k-1}} (\gamma_{n-k+2}-\gamma_{n-k+1}), \ \ \ n \geq k, \\
\tilde{\gamma}_{n} &=&  \gamma _{n-k+1}, \ \ \ n \geq k.
\end{eqnarray*}

\begin{exm}
Let  $\{P_{n}\}_{n\geq0}$ be the sequence of either monic Chebyshev polynomials of third kind $\{\tilde{V}_{n}\}_{n\geq0}$,  orthogonal with respect to  $d \mu(x) = (1+x)^{1/2} (1-x)^{-1/2} dx$ on $(-1,1)$,
or monic Chebyshev polynomials of fourth kind $\{\tilde{W}_{n}\}_{n\geq0}$,  orthogonal with respect to  $d \mu(x) = (1-x)^{1/2} (1+x)^{-1/2} dx$ on $(-1,1)$. In both cases there exists a representation
\begin{equation*}
   P_{n}(x) = \tilde{U}_{n}(x) + a \tilde{U}_{n-1}(x), \ \  n \geq 1,
\end{equation*}
where the coefficient $a$ depends on the choice of $P_{1}.$ For the Chebyshev polynomials of the third kind, with $\tilde{V}_{1}(x) = x - 1/2$, $a=-1/2$, and for the Chebyshev polynomials of the fourth kind,
with $\tilde{W}_{1}(x) = x + 1/2$, $a=1/2$, (see \cite[p.89]{14}).

Then
\begin{eqnarray*}
Q_{n}(x)   & = & P_{n}(x) + b_{1,n}P_{n-1}(x) + \cdots + b_{k-1,n}P_{n-k+1}(x) \\
    & = &  \tilde{U}_{n}(x) + (a+ b_{1,n}) \tilde{U}_{n-1}(x) + (a b_{1,n} + b_{2,n}) \tilde{U}_{n-2}(x) +  \cdots \\
  & & + (a b_{k-2,n}+b_{k-1,n})  \tilde{U}_{n-k+1}(x) + a b_{k-1,n} \tilde{U}_{n-k}(x).
\end{eqnarray*}

Thus, this problem is reduced to the one concerning Chebyshev polynomials of the second kind.

Notice that if  $b_{1,n} = b_{1}$ and $b_{2,n} = b_{2}$, for $n \geq k$, according to Example \ref{example_Un}, this yields
\begin{eqnarray*}
a b_{i,n}+ b_{i+1,n} &=& \tilde{b}_{i+1} \  \ 2 \leq i \leq k-2, \\
a b_{k-1,n} &=& \tilde{b}_{k-1}.
\end{eqnarray*}
The same analysis applies when we assume $b_{k-1,n} = b_{k-1}$ and $b_{k-2,n} = b_{k-2}$, for $n \geq k$.
\end{exm}

\begin{exm}
Let  $\{P_{n}\}_{n\geq0}$ be the sequence of monic Laguerre polynomials  $\{\tilde{L}_{n}^{(\alpha)}\}_{n\geq0}$, orthogonal with respect to $d \mu(x) = x^{\alpha} e^{-x} dx $ on $(0,\infty)$, $\alpha>-1$.
In this situation, $\beta_{n}=2n+\alpha+1$ for  $n \geq 0$, and $\gamma_n= n(n+\alpha)$ for $n \geq 1$. Consider the case when $b_{1,n}=b_{1}$ and $b_{2,n}=b_{2}$, for $n \geq k$. It follows from $(\ref{eq_for_b2})$ that
\begin{equation*}
   \frac{b_{k-1,n}}{b_{k-1,n-1}}\gamma _{n-k+1} =  \gamma _{n}
 + b_{1}\left( (2(n-1)+\alpha+1)- ( 2n+\alpha+1) \right),  \ \ n \geq k,
\end{equation*}
\begin{equation*}
   \frac{b_{k-1,n}}{b_{k-1,n-1}} =  \frac{\gamma_{n}-2 b_{1}}{\gamma _{n-k+1}}.  \  n \geq k,
\end{equation*}

\noindent \textbf{Step 1. } If $b_{1} = 0$, then
\begin{equation*}  
b_{k-1,n} =  \binom{n}{k-1} \binom{n+\alpha}{k-1} A(k,\alpha)  , \ \  n \geq k.
\end{equation*}
where $A(k,\alpha) $ does not depend on $n$. Therefore $b_{k-1,n}$ is a polynomial of degree $2k-2$ in $n$.

\noindent \textbf{Step 2. } If $b_{1} \neq 0$, then
\begin{eqnarray*}
 \frac{b_{k-1,n} }{b_{k-1,n-1} }
 &=&  \frac{(n-\alpha_{1})(n-\alpha_{2})}{(n-k+1)(n-k+1+\alpha) } ,
\end{eqnarray*}
where $\alpha_{1}$, $\alpha_{2}$ are, in general, complex numbers such that
$ (n-\alpha_{1})(n-\alpha_{2}) = n(n+\alpha) -2b_{1}$.   Thus,
\begin{eqnarray*}
    \frac{b_{k-1,n} }{b_{k-1,k-1}}
    & = & \frac{ \binom{n-\alpha_{1}}{n-k+1}  \binom{n-\alpha_{2}}{n-k+1}   }{\binom{n-k+1+\alpha}{n-k+1}}, \ \ n \geq k.
\end{eqnarray*}
Then $b_{k-1,n}$  is a rational function.

From $(\ref{eq_for_b1})$ we have
\begin{eqnarray*}
 \frac{b_{k-2,n}} {b_{k-1,n}} \gamma_{n-k+2} =  \frac{b_{k-2,n-1}} {b_{k-1,n-1}} \gamma_{n-k+1}  +2k-2.
\end{eqnarray*}
Then
\begin{eqnarray*}
 \frac{b_{k-2,n}} {b_{k-1,n}} \gamma_{n-k+2} =  (2k-2)n + c_{1},
\end{eqnarray*}
where $c_{1}$ does not depend on $n$, and
\begin{eqnarray*}
    b_{k-2,n}
 & = &  ((2k-2)n+c_{1}) \frac{ \binom{n-\alpha_{1}}{n-k+2} \binom{n-\alpha_{2}}{n-k+2} } {\binom{n-k+2+\alpha}{n-k+2} }
 \frac{b_{k-1,k-1} }{(k-1-\alpha_{1})(k-1-\alpha_{2})}.
\end{eqnarray*}
 Then $b_{k-2,n}$  is also a rational function.

Now we look at the behaviour of the coefficients $b_{i,n}$ for $3 \leq i \leq k-3$, and $n \geq k$.

From $(\ref{gamma_tilde_gamma})$ and $(\ref{eq_for_b_i})$ with $i=1$, we have
\begin{eqnarray*}
b_{3,n+1} & = &  b_{3,n} + b_{2}(\beta_{n-2}-\beta_{n}) - b_{1}^{2}(\beta_{n-1}-\beta_{n}) + b_{1}(\gamma_{n-1} - \gamma_{n} )\\
& = &  b_{3,n} - 4b_{2} + b_{1}(2b_{1} + (n-1)(n-1+\alpha) - n(n+\alpha)  ),
\end{eqnarray*}
we see that $ b_{3,n}  = c_{3,2} n^2 + c_{3,1} n + c_{3,0} $ is a polynomial of degree two in $n$.

Also, from $(\ref{gamma_tilde_gamma})$ and $(\ref{eq_for_b_i})$ with $i=2$, we have
\begin{eqnarray*}
b_{4,n+1} & = & b_{4,n} + b_{3,n}(\beta_{n-3}-\beta_{n}) + b_{2}\gamma_{n-2} - b_{2} \left(\gamma_{n} -2 b_{1} \right) \\
          & = & b_{4,n} - 6b_{3,n} + b_{2}\left(\gamma_{n-2}-\gamma_{n}\right) + 2b_{1}b_{2} \\
          & = & b_{4,n} - 6\left( c_{3,2} n^2 + c_{3,1} n + c_{3,0}\right) + b_{2}\left[ (n-2)(n-2+\alpha) -n(n+\alpha) \right] + 2b_{1}b_{2},
\end{eqnarray*}
and
$
b_{4,n}  = c_{4,3} n^3 + c_{4,2} n^2 + c_{4,1}n +c_{4,0}
$
is a polynomial of degree three in $n$.

Suppose that $k=5$. If $b_{1}=0$, then the above relations yield that $b_{k-1,n}=b_{4,n}$  is a polynomial of degree eight. On the other hand, $b_{4,n}$ is a polynomial of degree three, which is a contradiction. Otherwise, if $b_{1}\neq 0$, according to the above calculations, $b_{k-1,n}=b_{4,n}$ and $b_{k-2,n}=b_{3,n}$ are rational functions of the variable $n$. However, $b_{3,n}$ and $b_{4,n}$ are polynomials of degrees two and three, respectively. This is a contradiction again.

We conclude that it is not possible that $b_{1,n} \neq 0$ and $b_{2,n} \neq 0$, $n \geq k$, are constant real numbers when you deal with Laguerre orthogonal polynomials.
\end{exm}

\subsection{All constant coefficients}

Now we consider the special case when all the coefficients in (\ref{2})  do not depend on $n$.
Let us apply Theorem \ref{maintheorem} to obtain the necessary and sufficient
conditions for  the orthogonality of the monic polynomial sequence
$\left\{ Q_{n}\right\} _{n\geq 0}$. Let $\left\{ P_{n}\right\} _{n\geq 0}$ be a SMOP with respect to a linear functional  $u$ and
\begin{equation}
Q_{n}(x) =P_{n}(x) +b_{1}P_{n-1}(x)+\cdots+b_{k-1}P_{n-k+1}(x), \ \ n\geq k,  %
\label{constant}
\end{equation}%
where $\left\{ b_{i}\right\} _{i=1}^{k-1}$ are real numbers, and $b_{k-1} \neq 0$.
The above  necessary and sufficient conditions become
\begin{eqnarray}
\gamma _{n-k+1} -\gamma _{n} &=& b_{1}\left( \beta _{n-1}-\beta _{n}\right),\  n\geq k+1,
\label{constant1} \\
 b_{i-1}\left( \gamma _{n-k+1}-\gamma _{n-i+1}\right)\!&\!=\!&\!b_{i}\left( \beta
_{n-i}-\beta _{n}\right), \  n\geq k+1, \  2\leq i\leq k-1,
\label{constant2}
\end{eqnarray}%
\begin{eqnarray*}
\tilde{\beta}_{n}  &=&\beta _{n}, \ \  n\geq k+1, \\
\tilde{\gamma}_{n} &=&\gamma _{n-k+1}, \ \  n\geq k+1, \\ 
\tilde{\gamma}_{n} &=& \gamma _{n}+b_{1}\left( \beta _{n-1}-\beta _{n}\right)
\neq 0, \ \  n\geq k, \nonumber
\end{eqnarray*}%
where $\{ \tilde{\beta}_{n}\} _{n\geq 0}$ and $\{ \tilde{\gamma}_{n}\} _{n\geq 1}$ are the coefficients of the three term recurrence relation satisfied by the SMOP $\{Q_{n}\}_{n\geq 0}$, for $n\geq k$.

These results were obtained \cite{3}. In that paper the
authors provide also a detailed study of the case $k=3 $ with constant coefficients. The case $k=4$ with constant coefficients was analysed thoroughly in \cite{24}.

Now we focus our attention on the case when the sequence $\{ P_n \}_{n \geq 0}$ is symmetric, i.e., $\beta_n =0$, for all $n \geq 0$.  The conditions $\left( \ref{constant1}\right)$  and  $\left( \ref{constant2}\right) $ yield  the necessary and sufficient conditions,  for $n\geq k+1$,
\begin{eqnarray}
\gamma _{n-(k-1)} - \gamma _{n}  &=& 0, \label{symmetric1} \\
b_{i-1}\left( \gamma _{n-(k-1)}-\gamma _{n-(i-1)}\right)  &=& 0,\ \ 2\leq i\leq k-1. \label{symmetric2}
\end{eqnarray}%
Then, as a consequence of Theorem  \ref{maintheorem}, we obtain

\begin{cor} \label{corosymmetric}
Let $\left\{ P_{n}\right\} _{n\geq 0}$ be a symmetric monic polynomial sequence
and let $\left\{ Q_{n}\right\} _{n\geq 0}$ be a monic polynomial sequence
defined by relation $\left( \ref{constant}\right) $, for $n \geq k$. Then
$\left\{ Q_{n}\right\} _{n\geq 0}$ is a SMOP with recurrence coefficients
$\{ \tilde{\beta}_{n}\} _{n\geq 0}$ and $\left\{ \tilde{\gamma}_{n}\right\} _{n\geq 1}$ if and only if
the sequence $\{\gamma _{n}\}_{n\geq2}$ satisfies $\left( \ref{symmetric1}\right) $ and
$\left( \ref{symmetric2}\right) $. Furthermore, the recurrence coefficients  of SMOP
$\left\{ Q_{n}\right\} _{n\geq 0}$  satisfy $\tilde{\beta}_{n}  = 0$ and
$\tilde{\gamma}_{n} = \gamma _{n},$ for $n\geq k+1$.  In other words, $Q_{n}^{[k+1]}(x) = P_{n}^{[k+1]}(x)$,
where $Q_{n}^{[k+1]}$ and $P_{n}^{[k+1]}$ are the associated polynomials of order $k+1$ for the SMOP $\left\{ Q_{n}\right\} _{n\geq 0}$and $\{ P_n \}_{n \geq 0},$  respectively (see \cite{14}).
\end{cor}

Our next  result characterizes $\{\gamma _{n}\}_{n\geq2}$ as a periodic sequence and we  also discuss its possible periods.

\begin{thm}
Under the hypothesis of Corollary \ref{corosymmetric},  the sequence of the coefficients of the three-term recurrence relation  $\{\gamma _{n}\}_{n\geq2}$ must be a periodic sequence with period $j$, where $j$ is a divisor of $k-1$. Furthermore, if $|b_{j}|+|b_{k-1-j}| = 0$, for $1\leq j \leq \lfloor (k-1)/2 \rfloor$, then the period of the sequence $\{\gamma _{n}\}_{n\geq2}$ is $k-1$. If $(|b_{r}|+|b_{k-1-r}|)(|b_{s}|+|b_{k-1-s}|) \ldots (|b_{t}|+|b_{k-1-t}|) \neq 0$, for any $r,s,...,t$, such that $1\leq r,s,...,t \leq \lfloor (k-1)/2 \rfloor$, then the period of the sequence $\{\gamma _{n}\}_{n\geq2}$  is the greatest common divisor of $r,s,...,t$, and $k-1$.
\label{Theorem3.2}
\end{thm}

\begin{proof}
Conditions $\left( \ref{symmetric1}\right) $ and $\left( \ref{symmetric2}\right)$,  for $n\geq k+1$,   tell us that
if any coefficient $b_{j} \neq 0$, for  $1\leq j\leq k-2$, then
$ \gamma _{n-j}   =  \gamma _{n-(k-1)}  =  \gamma _{n}. $
Hence, we conclude that \\
$\bullet$ if any coefficient $b_{j} \neq 0$, for  $1\leq j \leq \lfloor (k-1)/2 \rfloor$, then
$ \gamma _{n-j} =  \gamma _{n}$ implies that $\{\gamma _{n}\}_{n\geq2}$ is a periodic sequence with period $j$; \\
$\bullet$ if any coefficient $b_{k-1-j} \neq 0$, for   $1\leq j\leq \lfloor (k-1)/2 \rfloor$, then
$\gamma _{n-(k-1-j)} =  \gamma _{n-(k-1)}$ implies that $\{\gamma _{n}\}_{n\geq2}$ is a periodic sequence with period $j$. 

As a summary, if $|b_{j}|+|b_{k-1-j}| \neq 0$, for any  $1\leq j\leq \lfloor (k-1)/2 \rfloor$, then
$\{\gamma _{n}\}_{n\geq2}$ is a $j$-periodic sequence.

The condition $\left( \ref{symmetric1}\right) $, i.e., $ \gamma _{n-(k-1)} = \gamma _{n}$,  for $n\geq k+1$, tell us that the sequence $\{\gamma _{n}\}_{n\geq2}$ is also a $k-1$-periodic sequence.

It is easy to see that a periodic sequence with both period $k-1$ and $j \leq \lfloor (k-1)/2 \rfloor$ has, in fact, period equals to the greatest common divisor of $k-1$ and $j$.  Since all divisors of $k-1$ but itself are included in $1\leq j\leq \lfloor (k-1)/2 \rfloor$, all choices of $b_{j}$ such that $|b_{j}|+|b_{k-1-j}| \neq 0$ yield the divisors in $1\leq j\leq \lfloor (k-1)/2 \rfloor$. Also the choice  $|b_{j}|+|b_{k-1-j}| = 0$ for $1\leq j\leq \lfloor (k-1)/2 \rfloor$ yields $k-1$ as the period.
\end{proof}

\begin{rem} i) If $|b_{j}|+|b_{k-1-j}| \neq 0$  for only one $j$ such that $1\leq j \leq \lfloor (k-1)/2 \rfloor$, then if $k-1$ is a multiple of $j$, the period of the sequence $\{\gamma _{n}\}_{n\geq2}$ is exactly $j$. \\
ii) Observe that to choose values for $|b_{r}|$ and $|b_{k-1-r}|$ one needs $k \geq 2r+1$.  \\
iii) Notice that the coefficient $\gamma_{1} > 0$ is free.
\end{rem}

\begin{rem}
If we consider a SMOP $\{ P_n \}_{n\geq0}$, such that $\beta _{n} = \beta$, for $n\geq0$, the conditions
$(\ref{constant1})$ and $(\ref{constant2})$ yield the same behaviour for $\{\gamma_{n}\}_{n\geq2}$  as in Theorem
$\ref{Theorem3.2}$ taking into account that it represents a shift in the variable for a symmetric SMOP.
\end{rem}

For the case $k=4$,  when the sequence $\{ P_n \}_{n\geq0}$
is not symmetric, in \cite{24} the authors also consider the choice $b_1=b_2=0$ and they prove that both sequences
$\{\gamma_{n}\}_{n=2}^{\infty}$ and $\{\beta_{n}\}_{n=2}^{\infty}$ must be $3$-periodic.
When one considers either only $b_{1} \neq 0$ or only $b_{2} \neq 0$, the behaviour of $\{\gamma_{n}\}_{n=2}^{\infty}$  and $\{\beta_{n}\}_{n=2}^{\infty}$ is  one-periodic.  Finally, with both $b_{1} \neq 0$ and $b_{2} \neq 0$ the behaviour of  $\{\gamma_{n}\}_{n=2}^{\infty}$ and $\{\beta_{n}\}_{n=2}^{\infty}$ depends on the values of $b_1$, $b_2$ and $b_3$.

\begin{rem}
Grinshpun \cite{Grinshpun2004} showed that Bernstein-Szeg\H{o}'s orthonormal polynomials of $i$-th kind, $i=1,2,3,4,$
and only them, can be represented as a linear combination of Chebyshev orthonormal polynomials of $i$-th kind,
respectively, with constant coefficients, namely
\begin{equation*}
\hat{Q}_{n}(x) =  \sum \limits_{j=0}^{k-1} t_{j} \hat{P}_{n-j}(x), \  \ n\geq k,
\end{equation*}
where $\{\hat{Q}_{n}\}_{ n\geq 0}$ denote the  Bernstein-Szeg\H{o} orthonormal polynomials of $i$-th  kind  and
$\{\hat{P}_{n} \}_{n\geq0}$ are the Chebyshev orthonormal polynomials of $i$th kind.

Sequences  of Bernstein-Szeg\H{o} polynomials are orthogonal with respect to the weight functions
\begin{equation*}
\omega_i(x) = \frac{\mu_i(x)}{\sigma_{k-1}(x)}, \ \ i=1,2,3,4,
\end{equation*}
where $\mu_i(x)$ is the Chebyshev weight function of the $i$th kind, $i=1,2,3,4,$ and  $\sigma_{k-1}(x)$ is a positive polynomial of degree $k-1$ on $(-1,1)$. The constants $t_{j}$ are given as the real coefficients of a polynomial $t(z)$ of degree $k-1$, that appears as the Fej\'{e}r-normalized representation of the positive polynomials $\sigma_{k-1}(x).$
Moreover, Grinshpun proves that if  $\left \{ P_{n}\right \}_{n\geq0}$ are the classical Chebyshev orthonormal  polynomials of one of the four kinds,  $\hat{Q}_{n}(x) =  \sum_{j=0}^{k-1} b_{j}\hat{P}_{n-j}(x), n\geq k,$
with $b_{0}b_{k-1}\neq0,$  and the polynomial $g(z)= \sum_{j=0}^{k-1} b_{j} z^{j}$ either does not have any zeros in the unit disc or all its zeros are located on the unit circle, then  either $\hat{Q}_{n}(x)$ or $\hat{Q}^{*}_{n}(x) =  \sum_{j=0}^{k-1} b_{j}\hat{P}_{n- k+1+j}(x), n\geq k$, are Bernstein-Szeg\H{o} polynomials of the corresponding kind.
\end{rem}

\textbf{Acknowledgements.}  We thank Dr.\,D.\,K. Dimitrov by his continued support. His comments and criticism have contributed to improve the presentation of the manuscript.


\bibliographystyle{amsplain}

\end{document}